\documentclass[ejs]{imsart}

\doi{10.1214/154957804100000000}
\pubyear{0000}
\volume{0}
\firstpage{1}
\lastpage{1}

\usepackage{amsmath,amssymb,amsthm}
\usepackage{bm}
\usepackage[numbers]{natbib}
\usepackage[colorlinks,citecolor=blue,urlcolor=blue]{hyperref}
\usepackage{graphicx,here,comment}
\usepackage{algorithm,algorithmic}
\usepackage{xcolor}
\usepackage{multirow}

\newtheorem{theorem}{Theorem}
\newtheorem{corollary}[theorem]{Corollary}
\newtheorem{lemma}[theorem]{Lemma}
\newtheorem{proposition}[theorem]{Proposition}
\newtheorem{assumption}{Assumption}
\theoremstyle{definition}
\newtheorem{definition}{Definition}
\newtheorem{remark}{Remark}

\newcommand{\R}{\mathbb{R}}
\newcommand{\N}{\mathbb{N}}

\newcommand{\mT}{\mathcal{T}}

\renewcommand{\Pr}{\mathbb{P}}

\renewcommand{\hat}{\widehat}
\renewcommand{\tilde}{\widetilde}

\newcommand{\mone}{\textbf{1}}

\renewcommand{\epsilon}{\varepsilon}
\DeclareMathOperator{\Cov}{Cov}

\usepackage{color}

\usepackage{newtxtext,newtxmath}
\usepackage{caption}
\usepackage{mathtools}
\mathtoolsset{showonlyrefs=true}

\begin{document}

\begin{frontmatter}

\title{Uniform Confidence Band for Optimal Transport Map on One-Dimensional Data\support{DP is financially supported by the Faculty of Science, Chiang Mai University, Thailand. RO was supported by Grant-in-Aid for JSPS Fellows (22J21512). MI was supported by JSPS KAKENHI (21K11780), JST CREST (JPMJCR21D2), and JST FOREST (JPMJFR216I).}}

\runtitle{Uniform Confidence Band for OT Map on 1D Data}

\author{\fnms{Donlapark} \snm{Ponnoprat}\ead[label=e1]{donlapark.p@cmu.ac.th}}
\address{Chiang Mai University\\\printead{e1}}

\author{\fnms{Ryo} \snm{Okano}\ead[label=e2]{okano-ryo1134@g.ecc.u-tokyo.ac.jp}}
\address{The University of Tokyo\\\printead{e2}}

\author{\fnms{Masaaki} \snm{Imaizumi}\ead[label=e3]{imaizumi@g.ecc.u-tokyo.ac.jp}}
\address{The University of Tokyo / RIKEN Center for Advanced Intelligence Project\\\printead{e3}}

\runauthor{Donlapark, Okano, and Imaizumi}

\begin{abstract}
    We develop a statistical inference method for an optimal transport map between distributions on real numbers with uniform confidence bands. The concept of optimal transport (OT) is used to measure distances between distributions, and OT maps are used to construct the distance. OT has been applied in many fields in recent years, and its statistical properties have attracted much interest. In particular, since the OT map is a function, a uniform norm-based statistical inference is significant for visualization and interpretation. In this study, we derive a limit distribution of a uniform norm of an estimation error for the OT map, and then develop a uniform confidence band based on it. In addition to our limit theorem, we develop a bootstrap method with kernel smoothing, then also derive its validation and guarantee on an asymptotic coverage probability of the confidence band. Our proof is based on the functional delta method and the representation of OT maps on the reals.
\end{abstract}

\end{frontmatter}

\allowdisplaybreaks

\section{Introduction}
We consider the framework of optimal transport (OT) and develop a statistical inference method on an OT map used in the concept.
Specifically, we consider the Monge problem; for probability measure $P$ and $Q$ on $\R$, we consider the following optimization problem
\begin{align}
    \min_{T \in \mT(P,Q)} \int_\R h(|x - T(x)|) dP(x), \label{def:monge}
\end{align}
where $h:\R \to \R$ is a nonnegative convex function (where $h(r)=r$ or $h(r)=r^2$ in most applications) and $\mT(P,Q)$ is a set of measurable maps $T: \R \to \R$ such that $P(T^{-1}(A)) = Q(A)$ for any measurable $A \subset \R$.
Let $T_0 \in \mT(P,Q)$ be an OT map which is the minimizer of the problem \eqref{def:monge}.
Our goal is the statistical inference on the OT map from samples: we construct a confidence band for $T_0$ based on samples independently generated from $P$ and $Q$, respectively. 
To this end, we develop estimators for the OT map estimator and derive a limiting distribution of their estimation error.

OT is a general framework to measure the distance between probability measures, and has attracted attention in a wide range of fields related to data analysis, such as social science, statistics, machine learning, and image analysis \cite{arjovsky2017wasserstein,galichon2018optimal,torous2021optimal}. 
For a general textbook, see \cite{villani2009optimal,peyre2019computational,villani2021topics}. 
The formulation of OT is given by the Monge problem \eqref{def:monge} or the Kantorovich problem related to the Wasserstein distance \cite{kantorovich1960mathematical}, in which the OT map plays an important role. 
OT plays a major role in modern data science due to its flexible extensibility and adaptability to high-dimensional data.

Statistical analysis for OT has played an important role in the usage of OT with finite samples, which are useful in assessing the uncertainty of estimated elements of the OT framework.
The most representative analysis examines the sample complexity of estimation methods for OT.
One of the typical interests is the problem of estimating the sum of optimized transport costs, which corresponds to the Wasserstein distance or its variants, and numerous studies have proposed optimal estimation methods \cite{weed2019sharp,genevay2019sample,niles2022estimation}.
In recent years, several studies \cite{hutter2021minimax,deb2021rates} have developed estimation methods for OT maps and also investigated their theoretical accuracy. 
On the other hand, statistical inference for OT, i.e., statistical tests and confidence sets, has also attracted attention. 
In addition to inference methods on OT costs and distances \cite{del2019central,sommerfeld2018inference,okano2022inference}, \cite{goldfeld2022limit,sadhu2023limit,manole2023central} propose inference methods on OT maps in several settings.

Despite the importance, the statistical inference for OT maps is still a developing issue, even when the sample is one-dimensional.
This is because that OT maps are functions and hence it is nontrivial to handle its uncertainty properly in a function space.
In particular, when investigating a global shape of an OT map (e.g., monotonicity) in some application fields, it is necessary to examine errors of the OT map estimator in the uniform norm sense, which is a nontrivial challenge.

In this paper, we develop a framework of confidence bands for an OT map between two unknown one-dimensional distributions. 
Specifically, we develop a \textit{uniform} confidence band that measures the OT map considering all input points in its support.
A uniform confidence band has the following advantages. 
First, it allows us to study the global shape of the OT map to be estimated. This advantage can not be achieved by arranging pointwise confidence intervals, which generally results in highly discontinuous curves. 
Second, it can be visualized on a plot and hence is easy to interpret, in contrast to confidence intervals based on the $L^2$-norm, which cannot be easily visualized.
Our approach is to use kernel smoothing on the distribution functions, and then perform a bootstrap scheme with kernel smoothing on the estimated distributions in order to construct a confidence band. As a theoretical contribution, we derive the asymptotic distribution of the estimation error in order to validate our confidence band.

Our contributions can be summarized as follows:
\begin{itemize}
    \item We develop the statistical inference on OT maps by proposing the uniform confidence bands that are computationally tractable.
    Coverage probabilities of the proposed bands are theoretically validated, and our simulations support the validity.
    \item The uniform confidence bands allow us to study a global property of functions, which can investigate complex hypotheses on OT maps.
    Our real-data experiments with population data demonstrate that our method rigorously tests a hypothesis whether the population pyramid changes faster, by examining a difference in shape between the OT map and a linear function.
\end{itemize}

As a technical contribution, we derive a simple asymptotic representation of the OT map estimation error as a linear sum of estimation errors for cumulative distribution functions. From this, we can apply the functional delta method to the representation, thereby obtaining the asymptotic distribution of the OT map estimation error.

\subsection{Related Studies}

Optimal transport has been actively studied in recent years. 
For a comprehensive review or textbook, see \cite{villani2009optimal,peyre2019computational,villani2021topics}.
It is known that there are several variations of the OT problem, such as the entropic regularization \cite{cuturi2013sinkhorn} or a low-dimensional projection named slicing \cite{bonneel2015sliced}.
Since there are numerous papers on OT, we refer to only a few that are relevant to our study.

Many studies have examined the sample complexity of the problem of estimating elements of OT from observed samples. 

A typical estimation target is a sum of optimized transport costs.
In the setting of one-dimensional data, \cite{bobkov2019one} gives a comprehensive analysis for the statistical properties of OT costs.
\cite{dudley1969speed,weed2019sharp, manole2021sharp} studied the estimation of the cost and reported that an optimal convergence rate of the estimation error is slow depending on the dimension of data, i.e., there is the curse of dimensionality. 
Related to this curse, \cite{niles2022estimation,genevay2019sample,lin2020projection,lei2020convergence} reported that the curse can be reduced by introducing the low-dimensional projection or the entropic regularization to the OT problem.
The estimation error of an OT map is also a target of intense study.
\cite{hutter2021minimax} has shown that an estimation error of OT maps also suffers from the curse of dimensionality of data, as is the case for the sum of costs. 
Similarly, entropic regularization and other techniques have been shown to mitigate this rate \cite{deb2021rates,manole2021plugin,pooladian2021entropic,divol2022optimal}.

For statistical inference, it is also a concern to derive limiting distributions of an error of estimators. 
Importantly, when data are multi-dimensional, it is not easy to obtain the limiting distributions relevant to the OT problem. 
Therefore, statistical inference is performed for the OT problem when the data are one-dimensional or discrete, or when some types of regularization are introduced.
For a limiting distribution of estimating a sum of transportation costs, \cite{munk1998nonparametric, ramdas2017wasserstein, del1999tests, del2005asymptotics} studied the one-dimensional case, and \cite{sommerfeld2018inference,bigot2019central, tameling2019empirical, klatt2020empirical,okano2022inference} studied the discrete data case.
\cite{mena2019statistical,del2022central,del2022improved,goldfeld2022statistical,manole2022minimax,imaizumi2022hypothesis} investigated a case with the regularization.
\cite{deo2023adaptive} constructs a confidence interval based on the OT distance in the problem of estimating density functions.
This situation is similar to the limiting distribution of estimating OT maps.
\cite{goldfeld2022limit} developed a statistical inference method on an OT map for the regularized OT problem case, and \cite{sadhu2023limit} derived a limiting distribution of an estimator for an OT map with a setting of semi-discrete data.
\cite{manole2023central} investigated the availability of limiting distributions in the multi-dimensional case and derived a pointwise convergence to a limiting distribution under certain conditions.

\subsection{Notation}

A \emph{cadlag function} is a right continuous function whose limits from the left exist everywhere.
$D[a,b]$ denotes the space of all cadlag functions $z:[a,b]\to [-1,1]$ equipped with the uniform norm. For a topological space $\Omega$, $BL_1(\Omega)$ denotes the space of Lipschitz functions $h:\Omega \to [-1,1]$ with Lipschitz constants bounded by one.
$\|\cdot\|$ denotes the Euclidean norm.
$\|f\|_{\infty} = \sup_t |f(t)|$ denotes the sup-norm.
$\ell^{\infty}(\mathcal{F})$ denotes the set of all uniformly bounded real functions on $\mathbb{R}$.
$\delta_x$ is the Dirac measure at $x$.
With an event $E$,  $\mone\{E\}$ is an indicator function.

For a distribution function $F:\Omega \to [0,1]$ with a support $\Omega \subset \R$ and $p \in [0,1]$, $F^{-1}(p) := \inf \{ x \in \Omega : p \leq F(x)\}$ denotes a quantile function of $F$.
Given probability measures $P$ and $Q$ on a support $\Omega \subseteq \R^d$, we define a set of transport maps $\mT(P,Q)$, which is a set of maps such that $T_{\sharp} P \coloneqq P(T^{-1}(\cdot)) = Q$. For random sequences $X_n$ and $Y_m$, the notation $(X_n,Y_m) = o_{P\otimes Q}(1)$ means $(X_n,Y_m)$ converges to $(0,0)$ in $P\otimes Q$-probability as $n,m\to\infty$ and possibly other conditions on $n$ and $m$ (see Section~\ref{sec:problem} below). 
$\xrightarrow{d}$ means convergence in distribution.
For non-random sequences of reals $\{x_n\}_{n \in \N}$ and $ \{y_n\}_{n \in \N}$, $x_n \asymp y_n$ denotes $\lim_{n \to \infty} x_n/y_n \to c$ with some $c  \in (0,\infty)$.

For any non-negative integer, $\beta$ and open set $U\subseteq \R$, $C^\beta(U)$ is the space of all bounded continuous real-valued functions that are $s$-times differentiable on $U$. This can be extended to the notion of H\"older space: $C^\beta(U)$ is the space of functions $f\in C^{\lfloor \beta \rfloor}(U)$ ($\lfloor \beta \rfloor$ is the integer part of $\beta$) whose $\lfloor \beta \rfloor$-th derivative is H\"older continuous with exponent $\beta-\lfloor \beta \rfloor$, that is, there is some constant $C>0$ such that for all $x,y\in U$, $\lvert D^{\lfloor \beta \rfloor}(x) - D^{\lfloor \beta \rfloor}(y) \rvert \leq C \vert x - y \rvert^{\beta - \lfloor \beta \rfloor}$ holds. 
We say that a function $f$ is (continuously) differentiable on $[a,b]$ if there is a small $\epsilon>0$ such that $f$ is (continuously) differentiable on $(a-\epsilon, b+\epsilon)$.

\subsection{Paper Organization}

Section \ref{sec:problem_setting} gives the optimal transport problem and its associated statistical inference problem.
Section \ref{sec:band} develops a uniform confidence band of the OT map as our proposed methodology.
Section \ref{sec:theory} shows theoretical guarantees of the developed confidence band with an outline of the proof.
Section \ref{sec:simulation} conducts a numerical simulation to show the experimental validity of the proposed confidence band.
Section \ref{sec:real_data} handles a real data analysis to demonstrate the usefulness of our method.
Section \ref{sec:conclusion} finally concludes our study.

\section{Problem Setting} \label{sec:problem_setting}

\subsection{Optimal Transport Problem with Samples}
\label{sec:problem}

We consider the setup with $d=1$.
Let $F_P$ and $F_Q$ be distribution functions of $P$ and $Q$, respectively. We shall make an assumption that $F_P$ is continuous, which guarantees a solution to the OT map problem~\eqref{def:monge} of the form (see e.g.~\cite[Remarks 2.19]{villani2021topics}):
\begin{align}
    T_0(x) = F_Q^{-1} \circ F_P (x). \label{eq:form_otmap}
\end{align}
Throughout the paper, we will only consider the values of the OT map $T_0$ on a closed interval $[a,b]$. For simplicity, we make an abuse of notation and treat $T_0$ and members of $\mT(P,Q)$ as functions on $[a,b]$.

Suppose that we observe $n$ i.i.d. samples $X_1,...,X_n \sim P$ and $m$ i.i.d. samples $Y_1,...,Y_m \sim Q$.
For simplicity, we assume that $n$ and $m$ follow the following asymptotic that often appears in nonparametric two-sample tests: with some $\kappa \in (0,1)$:
\begin{align}
    n/(n+m) \to \kappa . \label{def:ratio}
\end{align}

Our goal with this problem is to infer the true OT map $T_0$ using the samples. 
Specifically, we construct an estimator of $T_0$ based on the sample, as well as the following statistical inference.

\subsection{Statistical Inference via Confidence Band}

We aim to develop an estimator for the OT map $T_0$ and conduct statistical inference for $T_0$ from $n$ observations.
Specifically, with pre-specified $\alpha \in (0,1)$, we will develop a set $\mathcal{C}^{(\alpha)} = \{[\underline{c}(x), \overline{c}(x)] \mid x \in [a,b]\}$ with some functions $\underline{c}(x), \overline{c}(x)$ depending on the $n$ samples and $m$ samples and shows that
\begin{align}
    \Pr(T_0 \in \mathcal{C}^{(\alpha)}) = 1 - \alpha + o(1), \label{eq:validband}
\end{align}
as $n,m \to \infty$ with the limit ratio \eqref{def:ratio}.

Importantly, our interest is a confidence interval using the sup-norm, rigorously, our theory will show that $T_0(x) \in [\underline{c}(x), \overline{c}(x)]$ holds for every $x \in [a,b]$ with the asymptotic probability $1 - \alpha$.
The sup-norm is useful for intuitive analysis because it provides visualization in a functional space, whereas several common norms, e.g., the $L^2$-norm, cannot be visualized on a plot.

\section{Methodology}\label{sec:band}

In this section, we describe our methodology for constructing a confidence band for the OT map. First, we propose a kernel estimator for the OT map (Section~\ref{sec:kernelOT}). We then construct a confidence band based on this estimator and a bootstrap scheme (Section~\ref{sec:confband}).

\subsection{Kernel Estimator for OT Map}\label{sec:kernelOT}

Suppose that we observe the samples $X_1,...,X_n \sim P$ and $Y_1,\ldots,Y_m\sim Q$.
We estimate density and distribution functions from the samples, then use them to construct an estimator for the OT map $T_0$.

First, we define a kernel density estimator.
As a preparation, we define a kernel function $K:\R \to \R$, which should be a positive function and satisfy $\int K(t) dt = 1$.
There are several common choices: the Gaussian kernel, the Epanechinikov kernel, and so on (for an overview, see \cite{Silverman2018}).
Then, we define density estimators as
\begin{equation}\label{eq:kde}
    \hat{f}_P(x) \coloneqq \frac{1}{nr_n} \sum_{i=1}^n K\left( \frac{x - X_i}{r_n} \right) \quad \text{and} \quad \hat{f}_Q(y) \coloneqq \frac{1}{mr_m} \sum_{j=1}^m K\left( \frac{y - Y_j}{r_m} \right),
\end{equation}
where $r_n, r_m>0$ are bandwidth parameters dependent on $n$ and $m$, respectively. 
Second, we define estimators for distribution functions $F_P$ and $F_Q$ with $\overline{K}(x)\coloneqq \int_{-\infty}^x K(u) du$ as follows:
\begin{align}
    \hat{F}_P(x) &\coloneqq \int_{-\infty}^x \hat{f}_P(u) du = \frac{1}{n} \sum_{i=1}^n \overline{K}\left( \frac{x - X_i}{r_n} \right)
\end{align}
and
\begin{align}
    \hat{F}_Q(y) \coloneqq \int_{-\infty}^y \hat{f}_Q(u) du = \frac{1}{m} \sum_{j=1}^m \overline{K}\left( \frac{y - Y_j}{r_m} \right).
\end{align}
Since $\hat{f}_P$ and $\hat{f}_Q$ are strictly positive by the positive property of the kernel $K$, $\hat{F}_P$ and $\hat{F}_Q$ are strictly increasing function and hence invertible. 

We define an estimator for the OT map $T_0$ using the kernel estimators.
Using the form \eqref{eq:form_otmap}, we define the kernel smoothed estimator for the OT map $T_0$:
\begin{align}
    \hat{T}_{n,m}(x) = \hat{F}_Q^{-1} \circ \hat{F}_P (x).
\end{align}
This estimator has several advantages.
First, we can achieve a smooth estimation, and hence a smooth confidence band, for any input $x$. 
Second, we can guarantee the asymptotic validity of the uniform confidence band using smoothness. 
Third, as will be shown later, we can develop a bootstrap method with kernel smoothing for practical use and show its convergence.
We will compare this kernel-based approach with a pointwise estimator using empirical distributions in Remark \ref{remark:pointwise}.

\subsubsection{Assumption}

For the kernel estimation, we give the following assumptions on the distribution functions $F_P, F_Q$, and the kernel $K$:
\begin{assumption}\label{assumption:FGK}
    $F_P$ and $F_Q$ are differentiable on $\R$. For some $\beta\geq 0$, the densities $f_P\coloneqq F_P'$ and $f_Q \coloneqq F_Q'$ satisfy $f_P,f_Q \in C^\beta(\R)$.
    Further, the kernel $K$ is of order $> \beta + 1/2$, that is, $\int K(u)=1$ and $\int u^j K(u) du = 0$ holds for $j = 1,2,...,s$ for $s > \beta + 1/2$.
    Also, the bandwidth parameter is $r_n \asymp n^{-1/(2\beta+1)}$.
\end{assumption}
These conditions are widespread in density estimation using kernel methods. 
Regarding the order of the kernel, for example, \cite{tsybakov2008introduction} describes a method to construct kernels of arbitrary order using Legendre polynomials. 
The setup of the bandwidth parameter is commonly used as well, which is designed to balance bias and variance in density function estimation.
See \cite{tsybakov2008introduction} for details.

\begin{remark}[Choice of kernels and bandwidth]
    The choice of the kernel $K$ and bandwidth $r_n$ in Assumption \ref{assumption:FGK} is one of the methods used to implement \textit{under-smoothing}, which is common in constructions of confidence bands as summarized in Section 5.7 of \cite{wasserman2006all}.
In our design, we increase the order of the kernel $K$ from $\beta$ to $\beta + 1/2$ while keeping the bandwidth as $n^{-1/(2\beta+1)}$, the usual choice that leads to optimal rates of convergence in kernel density estimation.
This approach with the larger order kernels follows Corollary 2 of \cite{Gin2007}, which in turn allows us to utilize the functional delta method in the proof of our main theorem below.
\end{remark}

We further put the following assumption on a density function.

\begin{assumption}\label{assumption:gK}
    $f_Q$ is positive on $[F_Q^{-1}(F_P(a)),F_Q^{-1}(F_P(b))]$ and $K$ has bounded variation.
\end{assumption}
We remark that Assumption~\ref{assumption:gK} is more flexible than assuming that $f_Q$ is positive on $\R$ as it allows distributions that are supported on a proper subset of $\R$, such as gamma or chi-squared distributions.

\subsection{Construction of Confidence Band}\label{sec:confband}
We develop our methodology to construct a confidence band for $T_0$ based on the estimator $\hat{T}_{n,m}$.

\subsubsection{Bahadur Representation of Estimation Error}
Our basic strategy is a Gaussian approximation of the estimation error $\hat{T}_{n,m}(x) - T_0(x)$. 
Specifically, we investigate the following scaled estimation error
\begin{align}\label{eq:estconv_outline}
    \sqrt{n+m}\sup_x \frac{\big\lvert\hat{T}_{n,m}(x) - T_0(x)\big\rvert }{s_\kappa(x)},
\end{align}
then show that it converges in distribution to a limiting Gaussian process as $n,m \to \infty$, where $s_\kappa(x)$ is a standard deviation of $\hat{T}_{n,m}(x) - T_0(x)$. 
To achieve the Gaussian limit of \eqref{eq:estconv_outline}, the following asymptotic linear form, named the \textit{Bahadur representation}, plays an important role:
\begin{proposition}[Bahadur Representation]\label{prop:smootherror}
    Suppose that Assumption~\ref{assumption:FGK} and ~\ref{assumption:gK} hold.
    Then, for any $x\in [a,b]$, we have
    \begin{equation}\label{eq:smoothbahadur}
        \sqrt{n+m}\left(\hat{T}_{n,m}(x) - T_0(x)\right) = \frac{\sqrt{n+m}}{n} \sum_{i=1}^n \hat{\psi}(X_i,x) + \frac{\sqrt{n+m}}{m}\sum_{j=1}^m  \hat{\zeta}(Y_j,x)+o_{P\otimes Q}(1),
    \end{equation}
    where
    \begin{align}
        \hat{\psi}(X_i,x) &\coloneqq \frac{1}{f_Q(T_0(x))} \left[ \overline{K}\left(\frac{x - X_i}{r_n}\right) - F_P(x) \right],
    \end{align}
    and
    \begin{align}
        \hat{\zeta}(Y_j,x) \coloneqq \frac{1}{f_Q(T_0(x))}\left[F_P(x) - \overline{K}\left(\frac{T_0(x) - Y_j}{r_m}\right) \right].
    \end{align}
\end{proposition}
Since we assume that $Q$ has a density $f_Q$ that is nowhere zero, the representation holds for all $x\in [a,b]$.
With this asymptotic linear representation, we can guarantee the existence of Gaussian processes in the limit.
See Theorem~\ref{thm:consistency} below for a formal result.

\subsubsection{Plugin-Estimator for Standard Deviation}
To use the limiting Gaussian process in practice, we should derive several values.
First, we consider the standard deviation term $s_\kappa(x)$ of $\hat{T}_{n,m}(x) - T_0(x)$ appearing in \eqref{eq:estconv_outline}.
In view of~\eqref{eq:smoothbahadur}, we consider the following form of the standard deviation as
\begin{align}
    s_\kappa(x) &= \frac1{f_Q(T_0(x))}\sqrt{\left(\frac{1}{\kappa}+\frac1{1-\kappa}\right)F_P(x)(1-F_P(x))},
\end{align}
in which $F_Q(T_0(x)) = F_Q(F_Q^{-1}(F_P(x)))=F_P(x)$ plays an important role using the continuity of $F_Q$. 
In practice, we do not know $F_P,F_Q$, and $f_Q$, so it is impossible to compute $s_\kappa(x)$.  Instead, we estimate $s_\kappa(x)$ by the plug-in estimator:
\begin{align}
    \hat{s}_{n,m}(x) &= \frac1{\hat{f}_Q(\hat{T}_{n,m}(x))}\sqrt{\left( \frac{n+m}{n} + \frac{n+m}{m}\right)\hat{F}_P(x)(1-\hat{F}_P(x))}.
\end{align}
Note that the plug-in estimator always exists by the positivity of $\hat{f}_Q$.
The following result shows its consistency:
\begin{lemma}\label{lemma:sestimate}
    Suppose that Assumption~\ref{assumption:FGK} and ~\ref{assumption:gK} hold.
    Then, we have the following convergence:
    \begin{equation}\label{eq:sconv}
        \sup_{x\in [a,b]} \left\lvert \frac{\hat{s}_{n,m}(x)}{s_\kappa(x)}-1 \right\rvert \xrightarrow{P\otimes Q} 0.
    \end{equation}
\end{lemma}

\subsubsection{Bootstrap Approach with Kernel Smoothing and Confidence Band}

We approximate the Gaussian process for which \eqref{eq:estconv_outline} converges by a distribution generated by a bootstrap method.
Specifically, we develop a bootstrap method with kernel smoothing which newly generates samples from the estimated distribution functions $\hat{F}_P$ and $\hat{F}_Q$ by the smooth kernels.
In the bootstrap scheme, we sample $X^*_1,\ldots,X^*_n\sim \hat{F}_P$ and $Y^*_1,\ldots,Y^*_m\sim \hat{F}_Q$. 
Define bootstrap distribution functions $\hat{F}_P^*(x) \coloneqq n^{-1}\sum_{i=1}^n \mone\{X^*_i \leq x \}$ and $\hat{F}_Q^*(y) \coloneqq m^{-1}\sum_{j=1}^m \mone\{Y^*_j \leq y \}$. 
Then, we consider the bootstrap estimator for the OT map $T_0$ as
\begin{align}
    \hat{T}_{n,m}^* \coloneqq \hat{F}^{*-1}_Q(\hat{F}_P^*(x))
\end{align}
Note that $X_i^*$ and $Y_j^*$ are not subsamples of the dataset, but are generated from $\hat{F}_P$ and $\hat{F}_Q$.
This approach is more suitable when we apply the functional delta method to validate a confidence band in our proof.

Using the distribution of the bootstrap estimator $\hat{T}_{n,m}^* $, we derive quantiles of the distribution of the bootstrap version of the estimation error.
Let $\hat{P}_n$ and $\hat{Q}_m$ denote the conditional probability given $X_1,...,X_n$ and  $Y_1,...,Y_m$, respectively.
For any $\alpha \in (0,1)$ define 
\begin{align} \label{def:quantile}
    \hat{q}_{n,m}(1-\alpha) \coloneqq \text{\emph{the $(1-\alpha)$-quantile of }} \sqrt{n+m}\sup_x \frac{\big\lvert\hat{T}_{n,m}^*(x) - \hat{T}_{n,m}(x)\big\rvert}{\hat{s}_{n,m}(x)},
\end{align} 
under $\hat{\Pr}_{n,m} :=  \hat{P}_n \otimes \hat{Q}_m$.
Then, we propose the bootstrap confidence band
\begin{equation} \label{def:conf_band}
\mathcal{C}^{(\alpha)}_{n,m} (x)\coloneqq \left[\hat{T}_{n,m}(x) \pm \frac{\hat{s}_{n,m}(x)\hat{q}_{n,m}(1-\alpha)}{\sqrt{n+m}}\right], ~ x\in [a,b]
\end{equation}
Note that except for the bandwidth $r_n$, this confidence band is computed in a data-driven way.
Also, we will later propose a method to select $r_n$ based on the observed samples.

\begin{remark}[Comparison with a pointwise confidence interval] \label{remark:pointwise}
    Another natural approach is to construct a pointwise confidence interval for the OT map by plugging in the empirical distributions.
    To complement the main study, we present in Section \ref{sec:pointwiseinterval} a methodology for constructing a pointwise confidence interval, along with proof of its asymptotic validity and its empirical evaluations.
    Compared to this pointwise approach, our main proposed kernel-based method has several advantages. 
    First, it produces smooth confidence bands, which benefit interpretability. In contrast, pointwise confidence intervals based on the empirical distributions are riddled with discontinuities and can be difficult to interpret. 
    Second, the uniform confidence band can be used to infer the OT map across the whole domain. For instance, a $95\%$-level uniform confidence band has asymptotically $95\%$ chance to cover the true OT map over the domain. In contrast, a $95\%$-level pointwise confidence interval never contains all, but asymptotically only $95\%$ of the true OT map.
\end{remark}

\begin{remark}[Relation to ROC curves]
    We discuss a relation of the confidence band for OT maps to that for ROC (receiver operating characteristic) curves.
    ROC curves have a similar form $F_P \circ F_Q^{-1}$ and its confidence analysis has been developed by \cite{Horvath2008}.
    As a point of distinction between our study and the existing work, we develop the confidence band whose width differs for each input $x \in [a,b]$. 
    Rigorously, we have introduced the standard deviation $s_\kappa(x)$ and its estimator, then our confidence band achieves the adaptive widths for each input $x$.
    This result is in contrast to the confidence band by \cite{Horvath2008} for ROC curves, which has a constant width independent of $x$.
\end{remark}

\section{Theoretical Result} \label{sec:theory}

In this section, we present the main theoretical contributions of this paper, namely the bootstrap consistency (Theorem~\ref{thm:consistency}) and the asymptotic validity of the confidence band $\mathcal{C}_{n,m}^{(\alpha)}$ (Corollary~\ref{cor:validity_band}).

We first state the theorems in Section~\ref{sec:valband}, and then provide an outline of the proof of Theorem~\ref{thm:consistency} in Section~\ref{sec:proofoutline}. The full proofs of the theorems can be found in Appendix~\ref{sec:proofconsistency} and~\ref{sec:additionalproofs}.

\subsection{Validity of Confidence Band}\label{sec:valband}
We start with a consistency result of the bootstrap estimator. Notice the inclusions of the supremums on the left-hand sides of~\eqref{eq:estconv} and~\eqref{eq:empconv}, which are essential for obtaining a \emph{uniform} confidence band of $T_0$.
\begin{theorem}[Bootstrap consistency]\label{thm:consistency}
    Suppose that the distribution functions $F_P$ and $F_Q$ satisfy Assumption~\ref{assumption:FGK} and~\ref{assumption:gK}. Then, there are independent Brownian bridges $\mathbb{G}_1$ and $\mathbb{G}_2$ such that for any $x_0>0$, the followings holds as $n,m\to\infty$ and $n/(n+m)\to \kappa$:
    \begin{equation}\label{eq:estconv}
        \sqrt{n+m}\sup_x \frac{\big\lvert\hat{T}_{n,m}(x) - T_0(x)\big\rvert }{\hat{s}_{n,m}(x)}\xrightarrow{d} \sup_x \lvert \mathsf{Z}_\kappa(x) \rvert,
    \end{equation}
    and
    \begin{equation}\label{eq:empconv}
        \hat{\Pr}_{n,m}\left(\sqrt{n + m}\sup_x\frac{\big\lvert\hat{T}_{n,m}^*(x)-\hat{T}_{n,m}(x)\big\rvert}{\hat{s}_{n,m}(x)}\leq x_0\right) \xrightarrow{P\otimes Q} \mathbb{P}\Big(\sup_x\lvert \mathsf{Z}_\kappa(x) \rvert \leq x_0\Big),
    \end{equation}
    where
    \begin{align}
        \mathsf{Z}_\kappa(x) &= \frac1{f_Q(T_0(x))s_\kappa(x)}\left[\sqrt{1/\kappa}\mathbb{G}_1(F_P(x)) -\sqrt{1/(1-\kappa)}\mathbb{G}_2(F_Q(x))\right] \\
        &= \frac{\sqrt{1/\kappa}\mathbb{G}_1(F_P(x)) -\sqrt{1/(1-\kappa)}\mathbb{G}_2(F_P(x))}{\sqrt{(\kappa^{-1} + (1-\kappa)^{-1})F_P(x)(1-F_P(x))}}, \label{def:Z_kappa}
    \end{align}
    is a Gaussian process with unit variance for any $x\in \mathbb{R}$.
\end{theorem}
This theorem implies that the supremum values of both the scaled estimation error by the kernel estimator and the estimation error by the bootstrap estimator converge in distribution to the supremum of the same Gaussian process. 
In essence, the convergence \eqref{eq:estconv} is the intermediate result on the estimation error of the kernel method, and the convergence \eqref{eq:empconv} additionally provides the convergence of the bootstrap method.

Based on this result, we show the asymptotic validity of the proposed confidence band:
\begin{corollary}[Asymptotic Validity of Bootstrap Confidence Band] \label{cor:validity_band}
    Consider the proposed confidence band $\mathcal{C}^{(\alpha)}_{n,m}$ in \eqref{def:conf_band}.
    Then, under Assumption~\ref{assumption:FGK} and~\ref{assumption:gK}, the bootstrap confidence band $\mathcal{C}^{(\alpha)}_{n,m}$
    is asymptotically uniformly consistent at level $1-\alpha$, that is, it holds that
    \[\mathbb{P}\left(T_0(x) \in \mathcal{C}^{(\alpha)}_{n,m}(x), \forall x\in[a,b]\right) \to 1-\alpha,\]
    as $n,m\to\infty$ and $n/(n+m)\to \kappa$.
\end{corollary}

This result shows that our confidence bands are asymptotically valid in a uniform sense, that is, the OT map $T_0$ is included in our confidence band for every input $x$ simultaneously with the probability.

\subsection{Proof Outline of Theorem \ref{thm:consistency}}\label{sec:proofoutline}
Our proof consists of two parts: the first is the convergence of the estimated distributions by the kernel of the developed bootstrap method, and the second is the convergence of an application of the functional delta method. 
The details are described below.

\subsubsection{Convergence of Estimated Distributions}
As a preparation, we first derive limiting Gaussian processes of the distributions $ \hat{F}_P, \hat{F}_Q, \hat{F}_P^*$, and $\hat{F}_Q^*$.

We first describe the analysis of the distributions $\hat{F}_P^*$ and $\hat{F}_Q^*$ by the bootstrap method with kernel smoothing.
Rigorously, the central limit theorem in \cite{Komls1975} shows that there exists a Browninan bridge $\mathbb{G}_1$ such that the following holds for every $\delta > 0$:
    \begin{equation}
        \hat{P}_n\left(\sup_{x}\left\lvert \sqrt{n+m}(\hat{F}_P^*(x) - \hat{F}_P(x)) - \sqrt{1/\kappa}\mathbb{G}_1(F_P(x))\right\rvert >\delta\right) \xrightarrow{P\text{-a.s.}} 0,
    \end{equation}
as $n \to \infty$.
This shows that the error $\hat{F}_P^*(x) - \hat{F}_P(x)$ by the bootstrap with kernel smoothing converges to the Brownian bridge with the proper scaling in the uniform norm sense.
For technical reasons in further proof, we also derive the convergence in the sense of bounded Lipschitz metrics, that is, we obtain
    \begin{equation}
        \sup_{h\in BL_1(\ell^{\infty}(\mathcal{F}))}\left\lvert \mathbb{E}_n h(\sqrt{n+m}(\hat{F}_P^* - \hat{F}_P)) - \mathbb{E} h(\sqrt{1/\kappa}\mathbb{G}_1\circ F_P) \right\rvert \xrightarrow{P\text{-a.s.}} 0,
    \end{equation}
as $n \to \infty$. 
Here, we denote by $\mathbb{E}_n$ the expectation with respect to $\hat{P}_n$.
Similarly, we can obtain a similar limiting statement for the error $\hat{F}_Q^*(x) - \hat{F}_Q(x)$ by the other distribution by the bootstrap.

Next, we also analyze the error by the estimated distributions $\hat{F}_P$ and $\hat{F}_Q$ by the kernel method.
We apply the seminal analysis on the convergence of kernel convolutions \cite{Gin2007} and obtain the following  joint convergence:
\begin{equation}
    \sqrt{n+m}(\hat{F}_P - F_P, \hat{F}_Q - F_Q) \xrightarrow{d} (\sqrt{1/\kappa}\mathbb{G}_1\circ F_P,\sqrt{1/(1-\kappa)}\mathbb{G}_2\circ F_Q).
\end{equation}
as $n \to \infty$.
Here, $\mathbb{G}_1$ and $\mathbb{G}_2$ are some independent Brownian bridges.

\subsubsection{Functional Delta Method}
We study the convergence of the estimator $\hat{T}_{n,m}^*$ by using the above convergence results of the distributions and the representation \eqref{eq:form_otmap} of the OT map.
To the aim, we use the functional delta method (see Appendix~\ref{sec:tools} for a brief exposition).

Formally, we define a functional $\phi:D[a,b]\times D[a,b] \to D[a,b]$ as
\begin{align}
    \phi(u,v) = (v^{-1}\circ u(\cdot)),
\end{align}
which implies that $\hat{T}_{n,m}^* = \phi(\hat{F}_P^*, \hat{F}_Q^*)$, $\hat{T}_{n,m} = \phi(\hat{F}_P, \hat{F}_Q)$, and $T_0 = \phi(F_P, F_Q)$.
Then, we shall make a first-order approximation $\phi(\hat{F}_P,\hat{F}_Q) - \phi(F_P,F_Q) \approx \phi'(\hat{F}_P-F_P,\hat{F}_Q-F_Q)$. 
Thus it is important to first derive the \emph{Hadamard} derivative $\phi'$.
\begin{lemma}\label{lemma:hadamarddiff}
    Let $[a,b]$ satisfy the conditions in Proposition~\ref{prop:error}. Define the functional $\phi: D[a,b]\times D[a,b]\to D[a,b]$ by $\phi(u,v) = v^{-1}\circ u$.
    Then, $\phi$ is Hadamard differentiable at $(F_P,F_Q)$. Denoting $T_0=F_Q^{-1} \circ F_P$, the Hadamard derivative of $\phi$ at $(F_P,F_Q)$ is given by
    \begin{equation}\label{eq:derivative}
        \phi'(u, v) = \frac{1}{f_Q\circ T_0}\left[u - v\circ T_0\right].
    \end{equation}
\end{lemma}
We slightly extend this derivative for the design of confidence bands.
Define a functional $\Psi:D[a,b]\times D[a,b] \to D[a,b]$ as
    \begin{align}
        \Psi(u,v) = (v^{-1}\circ u(\cdot))/s_\kappa(\cdot).
    \end{align}
    Using Lemma \ref{lemma:hadamarddiff}, we derive its derivative as
     \begin{align}
         \Psi'(u,v) = (u(\cdot) - v \circ T_0(\cdot))/((f_Q\circ T)(\cdot) s_\kappa (\cdot)).
     \end{align}
     Note that we have added the term $s_\kappa(\cdot)$, which determines the scale of the confidence band.

Finally, we apply the functional delta method (Lemma \ref{lemma:deltaboot} in Appendix) and study the limit of the estimation error of the estimator $\hat{F}^*_{P}$ as
\begin{align}
    \sqrt{n+m}(\hat{T}_{n,m}^*-\hat{T}_{n,m})(\cdot)/s_\kappa(\cdot) &= \sqrt{n+m}(\Psi(\hat{F}^*_{P},\hat{F}^*_{Q}) - \Psi(\hat{F}_P,\hat{F}_Q)) \\
    &\xrightarrow{d} \Psi'(\sqrt{1/\kappa}\mathbb{G}_1\circ F_P, \sqrt{1/(1-\kappa)}\mathbb{G}_2\circ F_Q) \\
    &= \mathsf{Z}_\kappa,
\end{align}
where $Z_\kappa$ is defined in \eqref{def:Z_kappa}.
By a similar discussion, we also prove that the estimation error $\hat{T}_{n,m} - T_0$ of the kernel estimator also converges to the same Gaussian process.
In addition, we give evaluations of several plug-in estimators such as $\hat{s}_{n,m}(\cdot)$, then obtain the statement of Theorem \ref{thm:consistency}.

\section{Simulation} \label{sec:simulation}

\subsection{Simulation design}\label{sec:simdesign}

To support our asymptotic validity results, we perform a Monte Carlo simulation to evaluate the coverage probabilities of the confidence bands. For 1,000 iterations, we sample from two different probability distributions: $X_1,\ldots, X_n\sim N(0,1)$ and $Y_1,\ldots, Y_m\sim \operatorname{Gamma}(5, 0.5)$, where $n\in\{100,200,\ldots,2000\}$ and $m = n/4$ (that is, $\kappa=0.2$). In each iteration, we use 2,500 bootstrap samples to construct $(1-\alpha)$-level uniform confidence bands on the interval $[-2.5, 2.5]$, where $\alpha\in\{0.01,0.05,0.10\}$. The true OT map can be directly computed as $T_0=F_Q^{-1}\circ F_P$ where $F_P$ and $F_Q$ are the distribution function of $N(0,1)$ and $\operatorname{Gamma}(5,0.5)$, respectively. 
The Gaussian kernel is used for the uniform confidence bands with various smoothness parameters $\beta\in \{0.5, 1.0, 1.4\}$ and $r_n=\frac1{2}n^{-1/(2\beta+1)}$; so if we assume that the density functions are in $C^{\beta}(\R)$, then Assumption~\ref{assumption:FGK} and~\ref{assumption:gK} are satisfied.

To evaluate our confidence bands, we estimate the coverage probability as the proportion of 1000 runs in which $T_0(x)$ is contained in the confidence band \emph{for all} $x\in [-2.5, 2.5]$.
Additionally, we assess the confidence bands by calculating the median of the bands' average widths.

\subsection{Result and discussion}\label{sec:result}

\begin{table}[]
\begin{tabular}{crrcc}
\hline
$1-\alpha$              & $n$  & $m$ & Average width & Coverage probability \\ \hline
\multirow{3}{*}{0.90} & 200  & 50  & 29.13         & 0.656                \\
                      & 700  & 175 & 10.85         & 0.866                \\
                      & 2000 & 500 & 6.23          & \textbf{0.932}       \\ \hline
\multirow{3}{*}{0.95} & 200  & 50  & 35.67         & 0.668                \\
                      & 700  & 175 & 12.26         & 0.888                \\
                      & 2000 & 500 & 7.68          & \textbf{0.960}       \\ \hline
\multirow{3}{*}{0.99} & 200  & 50  & 40.94         & 0.737                \\
                      & 700  & 175 & 15.13         & 0.951                \\
                      & 2000 & 500 & 10.77         & \textbf{0.989}       \\ \hline
\end{tabular}
\caption{Evaluations of our $(1-\alpha)$-level uniform confidence bands of the optimal transport map from $N(0,1)$ to $\operatorname{Gamma}(5,0.5)$ based on 1,000 pairs of samples from each distribution. The table displays the median of average widths and the coverage probabilities of the confidence bands on $[-2.5,2.5]$. \label{table:uniform}}
\end{table}

The median of per-iteration average widths and the coverage probabilities over $x\in[-2.5,2.5]$ for $\beta=0.5$ and $n=200,700$ and $2000$ are shown in Table~\ref{table:uniform}. From the table, we can see that the coverage probabilities approach the nominal probabilities ($1-\alpha$), and the widths become smaller as $n$ increases. In particular, when $n=2000$ and $\alpha=0.90$ or $0.95$, the coverage probabilities are slightly larger than the nominal probabilities.

\begin{figure}
    \centering
    \includegraphics[width=0.95\hsize]{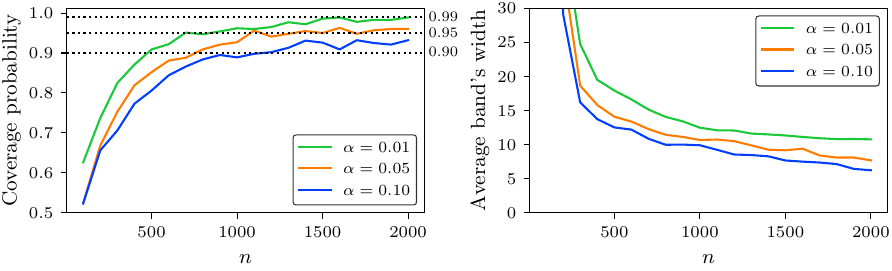}
    \caption{Plots of the coverage probabilities (left) and the median of average widths (right) of the simulated uniform confidence bands on $[-2.5,2.5]$ as functions of sample size $n$.}
    \label{fig:uniform1}
\end{figure}

The plots in Figure~\ref{fig:uniform1} illustrate the coverage probability and median of average width as functions of $n$. These plots lead us to the same conclusion: as $n$ increases, the average coverage probabilities approach the nominal probabilities, and the width of the band decreases.

\begin{figure}
    \centering\includegraphics[width=0.95\hsize]{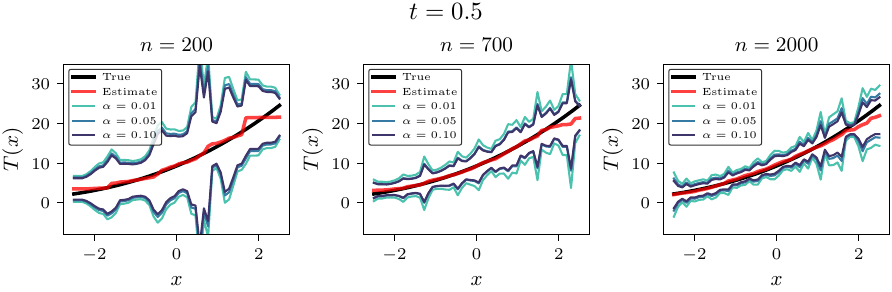}\\ \includegraphics[width=0.95\hsize]{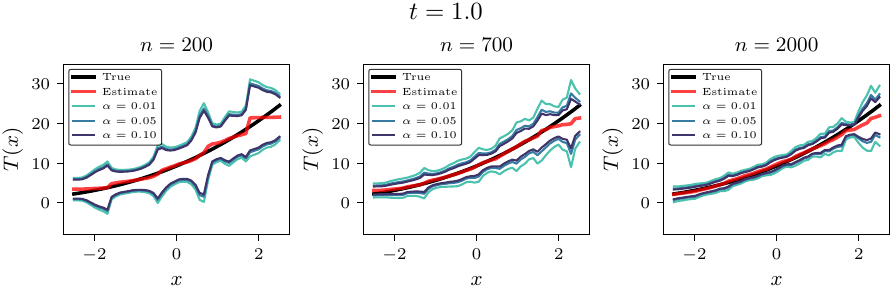}\\ \includegraphics[width=0.95\hsize]{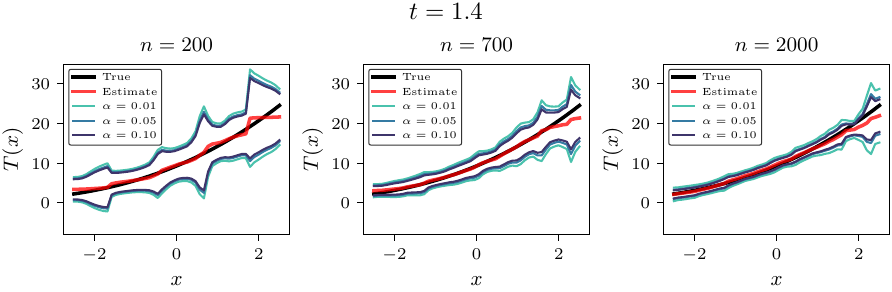}
    \caption{Examples of $(1-\alpha)$-level uniform confidence bands on $[-2.5,2.5]$, for three different values of $\alpha$ and three different values of the bandwidth parameter $\beta$, based on specific samples of size $200, 700$ and $2000$.}
    \label{fig:uniform2}
\end{figure}

We now examine the uniform confidence bands of a specific sample with $\beta=0.5, 1.0$ and $1.4$ (recall Assumption~\ref{assumption:FGK} that $\beta+0.5$ must be less than the order of the Gaussian kernel, which is $2$). The plots of the true optimal transport maps, their kernel estimates, and the uniform confidence bands are shown in Figure~\ref{fig:uniform2}. We observe that for $x>1.8$, the estimated transport map (the red curve) remains significantly distant from the actual transport map (the black curve). This disparity arises due to the heavier right tail of $\operatorname{Gamma}(5,0.5)$ compared to that of $N(0,1)$. Consequently, there is an inadequate number of sample points on the right tail of $N(0,1)$ to estimate the transport map between the two distributions. 

As $\beta$ increases, the kernel bandwidths increase and the confidence bands become smoother. Note that if the actual density functions are rougher than $C^\beta(\mathbb{R})$, the kernel estimate and the confidence band might be too smooth.

\section{Real data analysis} \label{sec:real_data}

As an application, we use our confidence bands to assess the uncertainty of our estimate of the transport map of the distribution of ages of death in 2001 to those in 2021. The data of the age of deaths from 12 countries were taken from the Human Mortality Database~\cite{HMD}. For each country, let us simply denote the dataset from the year 2001 by $X$, and those from the year 2021 by $Y$. Let $\lvert X \rvert=n$ and $\lvert Y \rvert = m$. Assume that the observed age of deaths in 2001 and 2021 are sampled from two separate continuous probability distributions. Then there is some uncertainty in our estimate due to randomness in the sampling. 

To construct the estimators and confidence bands at level $1-\alpha$, we use the Gaussian kernel. Our choices of bandwidths are guided by our theory in Section~\ref{sec:band}. Recall from Assumption~\ref{assumption:FGK} that $r_n\approx n^{-1/(2\beta+1)}$ where $\beta+0.5$ must be less than the order of the kernel. Since the Gaussian kernel is of second order, any $\beta < 1.5$ is permissible. In particular, we choose $\beta=1.25$ so that $2\beta+1 = 3.5$; this leads to our choices of kernel bandwidths $r_n = 2n^{-1/3.5}$ for $X$, and $r_m = 2m^{-1/3.5}$ for $Y$. With these bandwidths, we use the method in Section~\ref{sec:band} to construct a kernel estimate of the optimal transport map and a $95\%$ uniform confidence band for each country.

\begin{figure}
    \centering
    \includegraphics[width=0.95\hsize]{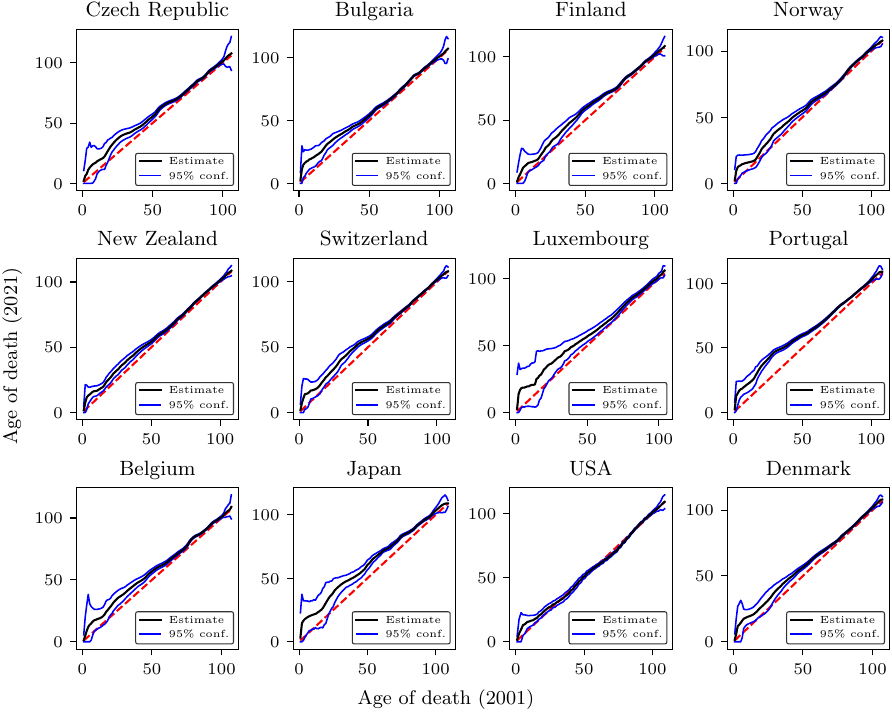}
    \caption{Analysis of the distribution shifts in ages of death from the year 2001 to 2021 using our uniform confidence bands.}
    \label{fig:mortality}
\end{figure}

The plots of our estimates and confidence bands for 12 countries are shown in Figure~\ref{fig:mortality}. The kernel estimate of the optimal transport map is the black curve, showing the correspondence between the age of deaths in 2001 and 2021. The identity function is the red dashed line. The estimate lying above the identity function indicates a shift in mortality towards higher ages. From these estimates and the confidence bands, we observe the most significant age shifts in Portugal (from $30$ to $45$) and Japan (from $30$ to $40$). Other countries also show minor age shifts, except for Norway, New Zealand, Luxembourg, and the USA, where we cannot confidently assert an upward shift in ages.

\section{Conclusion} \label{sec:conclusion}

In this paper, we develop a method to construct uniform confidence bands for the optimal transport maps based on two samples from two unknown continuous distributions. First, we use a kernel to estimate the densities, and then we use the empirical bootstrap to construct the confidence bands. We show that our confidence bands are asymptotically valid, meaning that they contain the actual OT map on an interval with a probability that approaches the nominal coverage probability.
We perform simulations to verify the validity of our confidence bands. As an application, we apply our confidence bands to analyze the shift in life expectancy across 12 countries from the year 2001 to 2021.

There are a couple of directions for future research. Firstly, our delta method and bootstrap procedure rely on first-order approximations. Exploring higher-order approximations would be an improvement worth considering. Secondly, our choice of kernel bandwidth directly follows from the theory of kernel density estimation, which, in practice, may not be sample-efficient in achieving consistency. This raises another research problem of finding a bandwidth search procedure that can achieve consistency more efficiently than the one presented in this paper.

\appendix

\section{Pointwise Confidence Intervals via the Empirical Distributions} \label{sec:pointwise}

In certain situations, one might wish to construct a confidence \emph{interval} for the value of the optimal transport at a specific point. Our approach to constructing such an interval follows closely to that of confidence bands, as the interval can be seen as a specific instance of confidence bands that covers only a single point. One key distinction is that, with only a single point, there is no need to estimate the standard deviation, and consequently removing the necessity for kernel density estimation.

\subsection{Methodology}

We first develop an empirical estimator for $T_0$.
We denote $\tilde{P}_n \coloneqq n^{-1} \sum_{i=1}^n \delta_{X_i}$ and $\tilde{Q}_m \coloneqq m^{-1} \sum_{j=1}^m \delta_{Y_i}$ as empirical measures with the observations.
With the empirical distribution functions $\tilde{F}_P(x) = n^{-1}\sum_{i=1}^n \mone\{X_i\leq x\}$ and $\tilde{F}_Q(y) = m^{-1}\sum_{j=1}^m \mone\{Y_j\leq y\}$, we define a plug-in estimator as the OT map from $\tilde{F}_P$ to $\tilde{F}_Q$:
\begin{align}
    \tilde{T}_{n,m}(x) = \tilde{F}_Q^{-1} \circ \tilde{F}_P (x).
\end{align}

\subsection{Bahadur representation}\label{sec:bahadur}

\begin{proposition}[Bahadur Representation of 1D Transport Map]\label{prop:error}
    Suppose that there exists an interval $[a,b]$ such that: (i) $F_Q$ is continuously differentiable on the interval $[F_Q^{-1}(F_P(a)),F_Q^{-1}(F_P(b))]$, (ii) $f_Q = F_Q'$ is nonzero on $[F_Q^{-1}(F_P(a)),F_Q^{-1}(F_P(b))]$. 
    Then, for any $x\in [a,b]$, we have
    \begin{align}
            &\sqrt{n+m}\left(\tilde{T}_{n,m}(x) - T_0(x)\right) \\
            &= \frac{\sqrt{n+m}}{f_Q(T_0(x))}\left[\frac1{n}\sum_{i=1}^n\mone\big\{X_i \leq x\big\} - \frac1{m}\sum_{j=1}^m\mone\big\{Y_j \leq T_0(x)\big\}\right] +o_{P\otimes Q}(1) \\
            &= \frac{\sqrt{n+m}}{n} \sum_{i=1}^n \psi(X_i,x) - \frac{\sqrt{n+m}}{m}\sum_{j=1}^m  \zeta(Y_j,x)+o_{P\otimes Q}(1),\label{eq:bahadur}
        \end{align}
    where
    \begin{align}
        \psi(X_i,x) \coloneqq \frac{1}{f_Q(T_0(x))} \left[ \mone\big\{X_i \leq x\big\} - F_P(x) \right],
    \end{align}
    and
    \begin{align}
        \zeta(Y_j,x)\coloneqq \frac{1}{f_Q(T_0(x))}\left[\mone\big\{Y_j \leq T_0(x)\big\} - F_P(x)\right].
    \end{align}
    Moreover, the $o_{P\otimes Q}(1)$ term does not depend on the choice of $x$.
\end{proposition}
The conditions on $F_Q$ imply that $F_Q$ is continuous on $[F_Q^{-1}(F_P(a)),F_Q^{-1}(F_P(b))]$, which implies $F_Q\circ T_0 = F_Q\circ F_Q^{-1}\circ F_P = F_P$ on $[a,b]$.

We derive an asymptotic representation of the estimation error $\tilde{T}_{n,m}(x) - T_0(x)$. 

\begin{proof}[Proof of Lemma \ref{lemma:hadamarddiff}]
    For $t \geq 0$, let $u_{t}\to u$ and $v_t\to v$ be functions in $D[a,b]$ such that $F_{P, t} = F_P+tu_t$ and $F_{Q, t} = F_Q + tv_t$ are in $D[a,b]$ for each $t$. To compute the Hadamard derivative, we consider the difference quotient:    
    \begin{equation}\label{eq:diffquot}
        \frac{F^{-1}_{Q, t} \circ F_{P,t} - F^{-1}_Q\circ F_P }{t} = \frac{F^{-1}_{Q,t} \circ F_t - G^{-1} \circ F_t }{t} + \frac{F_Q^{-1}\circ F_{P,t} - F_Q^{-1} \circ F_P }{t}
    \end{equation}
    We now apply the Taylor approximation for $F_{Q}^{-1}$. Since $F_Q'$ is continuous and nonzero on $[F_Q^{-1}(F_P(a)),F_Q^{-1}(F_P(b))]$, it follows from the inverse function theorem that $F_Q^{-1}$ is differentiable on $[F_Q(F_Q^{-1}(F_P(a))), F_Q(F_Q^{-1}(F_P(b)))] = [F_P(a),F_P(b)]$ and we have
    \begin{align}
        F_Q^{-1} ( F_{P,t}(x)) &= F_Q^{-1}(F_P(x) + tu_t(x)) \\
        &= F_Q^{-1}(F(x)) + \frac{1}{f_Q( F_Q^{-1}( F_P(x)))}tu_t(x) + o(tu_t(x)).
    \end{align}
    As $\lVert u_t \rVert_{\infty} \leq 1$, we have $o(tu_t(x)) = o(t)$ independent of $x$. From this, we approximate the second term on the right in~\eqref{eq:diffquot}.
    \begin{equation}\label{eq:diff2}
        \frac{F_Q^{-1}\circ F_{P,t} - F_Q^{-1}\circ F_P }{t} = \frac{1}{f_Q\circ F_Q^{-1}\circ F_P}u_t + o(1),
    \end{equation}
    as $t \to 0$.
    From~\cite[Lemma 21.3]{Vaart1998}, the Hadamard derivative of the quantile function $Q\mapsto F_Q^{-1}$ at $Q$ is $v \mapsto-\frac{v\circ F_Q^{-1}}{f_Q\circ F_Q^{-1}}$. From this, we use Taylor's formula again to obtain:
    \begin{align}
        F^{-1}_{Q,t} \circ F_{P,t} = F_Q^{-1} \circ F_{P,t} - \frac{v \circ F_Q^{-1}\circ F_{P,t}}{f_Q\circ F_Q^{-1}\circ F_{P,t}}t + o(t).
    \end{align}
    Therefore, the first term on the right of~\eqref{eq:diffquot} is
    \begin{equation}\label{eq:diff1}
        \frac{F^{-1}_{Q,t} \circ F_{P,t} - F_Q^{-1} \circ F_{P,t} }{t} = F_Q^{-1} \circ F_{P,t} - \frac{v \circ F_{Q}^{-1}\circ F_{P,t}}{f_Q\circ F_Q^{-1}\circ F_{P,t}} + o(1).
      \end{equation}
    Combining~\eqref{eq:diff2} and~\eqref{eq:diff1} and the continuity of $v \circ F_Q^{-1}$ and $g\circ F_Q^{-1}$, we conclude the convergence in the Hadamard sense, that is, it holds that
    \begin{align}
        \lim_{t\to 0}\frac{F^{-1}_{Q,t} \circ F_{P,t}  - F_Q^{-1} \circ F_P }{t} &=  \frac{1}{f_Q\circ F_Q^{-1}\circ F_P }\left[u - v \circ F_Q^{-1} \circ F_P\right] \\
        &= \frac{1}{f_Q\circ T}\left[u - v \circ T\right].
    \end{align}
\end{proof}
\begin{proof}[Proof of Proposition~\ref{prop:error}]
    We first recall that $F_Q$ has a nonzero derivative at $F_Q^{-1}(F_P(x))$ for any $x\in [a,b]$; this implies that $F_Q$ is locally invertible at $F_Q^{-1}(F_P(x))$, and so $F_Q(T_0(x)) = F_Q(F_Q^{-1}(F_P(x))) = F_P(x)$. This allows us to simplify the Hadamard derivative~\eqref{eq:derivative} evaluated at $u_i(x)=\mone\{X_i \leq x\}-F_P(x)$ and $v_j(y) = \mone\{Y_j \leq y\}-F_Q(y)$ as follows:
    \begin{align}
        \quad \phi'(u_i,v_j)(x) &= \frac1{f_Q(T_0(x))}\left[\mone\big\{X_i \leq x\big\}-F_P(x) - \mone\big\{Y_j \leq F_Q^{-1}( F_P(x))\big\} + F_Q(T_0(x))\right] \\
        &= \frac{1}{g(T_0(x))}\left[\mone\big\{X_i \leq x\big\}-F_P(x) - \mone\big\{Y_j \leq T_0(x)\big\} + F_P(x)\right] \\
        &= \psi(X_i,x) - \zeta(Y_j,x). 
    \end{align}
    Using the functional delta method (Lemma~\ref{lemma:functional_delta}) and the linearity of $\phi'$, we arrive at the final approximation.
    \begin{align}
        \sqrt{n+m}\left(\tilde{T}_{n,m}(x) - T_0(x)\right) &= \sqrt{n+m}\left(\phi(\tilde{F}_{P,n},\tilde{F}_{Q,m}) - \phi(F,G)\right) \\
        &= \sqrt{n+m}\phi'(\tilde{F}_{P,n}-F_P,\tilde{F}_{Q,m}-F_Q) + o_{P\otimes Q}(1) \\
        &= \frac{\sqrt{n+m}}{n}\sum_{i=1}^n\psi(X_i,x) + \frac{\sqrt{n+m}}{m}\sum_{j=1}^m\zeta(Y_j,x) + o_{P\otimes Q}(1).
    \end{align}
\end{proof}

\subsection{Gaussian Approximation Theorem}

With the Bahadur representation~\eqref{eq:bahadur}, we can easily develop a Gaussian approximation on the estimation error. If $F_Q$ an interval $[a,b]$ satisfy the conditions in Proposition~\ref{prop:error}, we obtain the following representation of the error on $[a,b]$:
\begin{align}
    \sqrt{n+m}\left(\tilde{T}_{n,m}(x) - T_0(x)\right) = \frac{\sqrt{n+m}}{n}\sum_{i=1}^n\psi(X_i,x) - \frac{\sqrt{n+m}}{m}\sum_{j=1}^m\zeta(Y_j,x) + o_{P\otimes Q}(1)\label{eq:decomp}
\end{align}
Note that the two processes are independent, since $\{X_i\}_{t = 1}^n$ and $\{Y_j\}_{j=1}^m$ are mutually independent.

We consider Gaussian limits of the terms in \eqref{eq:decomp}. In view of Proposition~\ref{prop:error}, Donsker's theorem tells us that there exist two independent Brownian bridges $\mathbb{G}_1$ and $\mathbb{G}_2$ such that the following convergences of $D[a,b]$ functions hold under the uniform norm:
\begin{align}
\frac1{\sqrt{n}}\sum_{i=1}^n\psi(X_i,x)  &= \frac1{f_Q(T_0(x))}\frac1{\sqrt{n}}\sum_{i=1}^n [\mone \{X_i \leq x\} - F_P(x)] \xrightarrow{d} \frac1{f_Q(T_0(x))}\mathbb{G}_1(F_P(x)), \label{eq:G1F}
\intertext{and}
\frac1{\sqrt{m}}\sum_{j=1}^m\xi(Y_j,x)  &= \frac1{f_Q(T_0(x))}\frac1{\sqrt{m}}\sum_{j=1}^m [\mone \{Y_j \leq T_0(x)\} - F_P(x)] \xrightarrow{d} \frac1{f_Q(T_0(x))}\mathbb{G}_2(F_P(x)), \label{eq:G2F}
\end{align}
on $[a,b]$. Recalling the sample size condition $n/(n+m)\to \kappa$, we have the following convergence to a Gaussian process:
\begin{align}
    \sqrt{n + m} (\tilde{T}_{n,m}(x) - T_0(x)) & =  \frac{\sqrt{n + m}}{n} \sum_{i=1}^n \psi(X_i; x) - \frac{\sqrt{n + m}}{m} \sum_{j=1}^m \zeta(Y_j; \cdot) + o_{P\otimes Q}(1) \\
    & \overset{d}{\to} \frac1{f_Q(T_0(x))}\left[\sqrt{1/\kappa} \mathbb{G}_1 (x) - \sqrt{1/(1-\kappa)}  \mathbb{G}_2 (F_P(x))\right].
\end{align}
We remark that the covariance function of $\mathbb{G}_1\circ F_P$ and $\mathbb{G}_2\circ F_P$ are can be written explicitly:
\begin{equation*}
    \Cov\left(\mathbb{G}_1(x), \mathbb{G}_1(x') \right) = \Cov\left(\mathbb{G}_1(x), \mathbb{G}_2(x') \right) =  \min\{F_P(x), F_P(x')\} - F_P(x) F_P(x').
\end{equation*}

\subsection{Bootstrap for Pointwise Confidence Intervals}\label{sec:pointwiseinterval}

Recall the empirical distribution functions $\tilde{F}_{P,n}(x) = n^{-1}\sum_{i=1}^n \mone\{X_i\leq x\}$ and $\tilde{F}_{Q,m}(y) = m^{-1}\sum_{j=1}^m \mone\{Y_j\leq y\}$. We define a plug-in estimator as the optimal transport map from $\tilde{P}_n$ to $\tilde{Q}_m$:

Let $X^*_1,\ldots,X^*_n \sim \tilde{P}_n$ and $Y^*_1,\ldots,Y^*_m \sim \tilde{Q}_m$. Define the bootstrap distribution function $\tilde{F}^*_{P,n}(x) \coloneqq n^{-1}\sum_{i=1}^n \mone\{X^*_i \leq x\}$ and $\tilde{F}^*_{Q,m}(y) \coloneqq m^{-1}\sum_{j=1}^m \mone\{Y^*_j \leq y\}$. The bootstrap transport map is then given by $\tilde{T}_{n,m}^*(x) \coloneqq \tilde{F}^{* -1}_{Q,m}(\tilde{F}_{P,n}^*(x))$.
\begin{theorem} \label{thm:confidence_pointwise}
Let $\Phi$ be the distribution function of the standard normal distribution. If $F_Q$ is continuously differentiable at $F_Q^{-1}(F_P(x))$ and the derivative $f_Q(F_Q^{-1}(F_P(x)))$ is nonzero, then
\begin{equation}\label{eq:condconsist}
\sup_{x_{0}}\left\lvert \tilde{P}^*_{n,m}\left(\sqrt{n+m}(\tilde{T}_{n,m}^{*}(x)-\tilde{T}_{n,m}(x)) \le x_0\right) -\Phi\left(x_0/\sigma_T(x)\right) \right\rvert \xrightarrow{P\otimes Q} 0, 
\end{equation}
where $\sigma_T(x) = (f_Q(T_0(x))^{-1}\sqrt{(\kappa^{-1}+(1-\kappa)^{-1})F_P(x)(1-F_P(x))}$.
\end{theorem}
\begin{proof}[Proof of Theorem \ref{thm:confidence_pointwise}]
    With $\sigma_P(x) \coloneqq \sqrt{F_P(x)(1-F_P(x))}$ and $\sigma_Q(x) \coloneqq \sqrt{F_Q(x)(1-F_Q(x))}$, we will show the weak conditional convergences of the bootstrap empirical distributions:
    \begin{align}
        &\sup_{x_0} \left\lvert \tilde{P}_n\left(\sqrt{n}[\tilde{F}^*_{P,n}(x) - \tilde{F}_{P,n}(x)] \leq x_0) - \Phi(x_0/\sigma_{P}(x)\right) \right\rvert \xrightarrow{P} 0, \label{eq:convU}
    \intertext{and}
        &\sup_{x_0} \left\lvert \tilde{Q}_m\left(\sqrt{m}[\tilde{F}^*_{Q,m}(x) - \tilde{F}_{Q,m}(x)] \leq x_0) - \Phi(x_0/\sigma_{Q}(x)\right) \right\rvert \xrightarrow{Q} 0. \label{eq:convV}
    \end{align}
    Observe that the variance of $\tilde{F}^*_{P,n}(x) - \tilde{F}_{P,n}(x)$ is
    \[ \sigma^2_{\tilde{P}_n}  \coloneqq \tilde{F}_{P,n}(x)(1-\tilde{F}_{P,n}(x)) = \sigma^2_P + o_P(1).\]
    Applying the Berry-Esseen on $\sqrt{n}[\tilde{F}^*_{P,n}(x) - \tilde{F}_{P,n}(x)]$ with respect to the empirical distribution $\tilde{P}_n$, there is a constant $A>0$ such that
    \begin{align}
        &\sup_{x_0} \left\lvert \tilde{P}_n\left(\sqrt{n}[\tilde{F}^*_{P,n}(x) - \tilde{F}_{P,n}(x)] \leq x_0) - \Phi(x_0/\sigma^2_{\tilde{P}_n}(x)\right) \right\rvert \\
        &\leq \frac{A\sum_{i=1}^n \lvert \tilde{F}^*_{P,n}(x) - \tilde{F}_{P,n}(x) \rvert^3}{n^{3/2}\sigma^2_{\tilde{P}_n}} \\
        &\leq \frac{A}{n^{1/2}\sigma_{\tilde{P}_{n}}} \\
        &= \frac{A}{n^{1/2}\sigma_P} + o_P(1).
    \end{align}
    We now find an upper bound of $ \lvert  \Phi(x_0/\sigma_{\tilde{P}_n}(x)) -  \Phi(x_0/\sigma_P(x)) \rvert$ via the method of calculus on the function $\sigma \mapsto \Phi(x_0/\sigma)$. By the boundedness of the Gaussian density, there is a constant $B$ such that
    \begin{equation*}
         \sup_{x_0} \left\lvert \frac{\partial \Phi (x_0/\sigma)}{\partial \sigma} \right\rvert  \leq \frac{B}{\sigma}.
    \end{equation*}
    It follows from the mean value theorem that
    \begin{equation*}
        \lvert  \Phi(x_0/\sigma_{\tilde{P}_n}(x)) -  \Phi(x_0/\sigma_P(x)) \rvert \leq \frac{B}{\min\{\sigma_{\tilde{P}_n}(x),\sigma_P(x)\}}\lvert \sigma_{\tilde{P}_n}(x)-\sigma_P(x)\rvert\xrightarrow{P} 0, 
    \end{equation*}
    which allows us to conclude~\eqref{eq:convU}. The convergence~\eqref{eq:convV} follows analogously.

    For any cadlag function $\Lambda$, define a random function $Z_{\Lambda}(x) \sim \mathcal{N}(0,\Lambda(x)(1-\Lambda(x)))$. By the weak law of large number,
    \begin{align}
        \sqrt{n}[\tilde{F}_{P,n}(x) - F_P(x)]\xrightarrow{d} Z_{F_P}(x), \label{eq:ZF}
        \intertext{and}
        \sqrt{m}[\tilde{F}_{Q,m}(x) - F_{Q}(x)]\xrightarrow{d} Z_{F_Q}(x). \label{eq:ZG}
    \end{align}

    Consider the functional $\phi(u,v) = v^{-1}\circ u$. From Lemma~\ref{lemma:hadamarddiff}, the Hadamard derivative of $\phi$ at $(F_P,F_Q)$ is $\phi'(u,v) = (u - v\circ T)/f_Q\circ T$. We now apply the delta method for bootstrap Lemma~\ref{lemma:deltaboot} on $\phi$: from~\eqref{eq:convU},~\eqref{eq:convV},~\eqref{eq:ZF},~\eqref{eq:ZG} and $n/(n+m) \to \kappa$, we have that $\sqrt{n+m}(\phi(\tilde{F}^*_{P,n},\tilde{F}^*_{Q,m}) - \phi(\tilde{F}_{P,n},\tilde{F}_{Q,m})) = \sqrt{n+m}( \tilde{T}_{n,m}^*-\tilde{T}_{n,m})$ converges in distribution to
    \begin{align}
        \phi'(\sqrt{1/\kappa}Z_{F_P}, \sqrt{1/(1-\kappa)}Z_{F_Q}) &=\frac1{f_Q\circ T_0}\left(\sqrt{1/\kappa}Z_{F_P} - \sqrt{1/(1-\kappa)}Z_{F_Q}\circ T_0\right) \\
        &=\frac1{f_Q\circ T_0}\left(\sqrt{1/\kappa}Z_{F_P} - \sqrt{1/(1-\kappa)}Z_{F_Q\circ T_0}\right) \\
        &= \frac1{f_Q\circ T_0}\left(\sqrt{1/\kappa}- \sqrt{1/(1-\kappa)}\right)Z_{F_P},
    \end{align}
    conditional to $X_1,\ldots,X_n$ in probability. In other words, the convergence~\eqref{eq:condconsist} holds true. 
\end{proof}
\begin{corollary}[Asymptotic Validity of the Bootstrap Confidence Interval] \label{cor:valid_point_conv}
    For any $\alpha \in (0,1)$, let $\alpha' \coloneqq \alpha/2$ and define 
    \[\tilde{q}_{n,m}(1-\alpha') \coloneqq \text{\emph{the $(1-\alpha')$-quantile of }} \sqrt{n+m}\Big(\tilde{T}_{n,m}^*(x) - \tilde{T}_{n,m}(x)\Big) \text{\emph{ under $\tilde{P}_{n,m}$}}.\] 
    Then, for any $x\in\R$ such that $F_Q$ is continuously differentiable at $F_Q^{-1}(F_P(x))$ and the derivative $f_Q(F_{Q}^{-1}(F_P(x)))$ is nonzero, the bootstrap confidence band for $T_0(x)$:
    \begin{equation}
        \mathcal{C}^{(\alpha)}_{n,m}(x) \coloneqq\left[\tilde{T}_{n,m}(x) \pm \frac{\tilde{q}_{n,m}(1-\alpha')}{\sqrt{n+m}}\right],
    \end{equation}
    is asymptotically consistent at level $1-\alpha$, that is, \[\mathbb{P}\left(T_0(x) \in \mathcal{C}^{(\alpha)}_{n,m}(x)\right) \to 1-\alpha,  \]
    as $n,m\to\infty$ and $n/(n+m)\to \kappa$.
\end{corollary}
\begin{proof}[Proof of Corollary \ref{cor:valid_point_conv}]
    For notational convenience, denote $\Phi_{T,x}(x_0) \coloneqq \Phi(x_0/\sigma_T(x))$. Recall from~\eqref{eq:condconsist} that the sequence $\tilde{P}_{n,m}\Big(\sqrt{n + m}\big(\tilde{T}_{n,m}^*(x)-\tilde{T}_{n,m}(x)\big)\leq x_0\Big)$ converges in $P\otimes Q$-probability to $\Phi\left(x_0/\sigma_T(x)\right)$. By passing along any subsequence and a further subsequence, we can assume that the convergence is $P\otimes Q$-almost surely. 
    
    The almost-sure convergence of the distribution functions implies the almost-sure convergence of the corresponding quantile functions~\cite[Lemma 21.2]{Vaart1998}. In particular, $\tilde{q}_{n,m}(1-\alpha')$, which is the $1-\alpha'$-quantiles of $\sqrt{n + m}\big(\tilde{T}_{n,m}^*(x)-\tilde{T}_{n,m}(x)\big)$, converges $P\otimes Q$-a.s. to $\Phi_{T,x}^{-1}(1-\alpha')$. It follows from Slutsky's theorem that $\sqrt{n + m}\big(\tilde{T}_{n,m}(x)-T_{n,m}(x)\big) - \tilde{q}_{n,m}(1-\alpha')$ converges in distribution to $Z_T(x) - \Phi_{T,x}^{-1}(1-\alpha')$, where $Z_T \sim \mathcal{N}(0,\sigma^2_T)$. Therefore, as $n,m\to\infty$ and $n/(n+m)\to \kappa$,
    \begin{align}
     &\mathbb{P}\left(T_0(x) \geq \tilde{T}_{n,m}(x) - \tilde{q}_{n,m}(1-\alpha')/\sqrt{n+m}\right) \\
     &= \mathbb{P}\left(\sqrt{n + m}\big(\tilde{T}_{n,m}(x)-T_{n,m}(x)\big) \leq \tilde{q}_{n,m}(1-\alpha')\right) \\
     &\to \mathbb{P}\left(Z_T \leq \Phi_{T,x}^{-1}(1-\alpha')\right) \\
     &= \Phi_{T,x}(\Phi_{T,x}^{-1}(1-\alpha')) \\
     &= 1-\alpha'.
    \end{align}
    Similarly, by the symmetry of $Z_T$,
    \begin{align}
     &\mathbb{P}\left(T_0(x) \leq \tilde{T}_{n,m}(x) + \tilde{q}_{n,m}(1-\alpha')/\sqrt{n+m}\right) \\
     &= \mathbb{P}\left(\sqrt{n + m}\big(\tilde{T}_{n,m}(x)-T_{n,m}(x)\big) \geq -\tilde{q}_{n,m}(1-\alpha')\right) \\
     &\to \mathbb{P}\left(Z_T \geq -\Phi_{T,x}^{-1}(1-\alpha')\right)  \\
     &= \mathbb{P}\left(Z_T \leq \Phi_{T,x}^{-1}(1-\alpha')\right) \\
     &= 1-\alpha' .
    \end{align}
    We conclude that $\mathbb{P}\left(T_0(x) \in \mathcal{C}^{(\alpha)}_{n,m}(x)\right) \to 1-2\alpha'=1-\alpha $.
\end{proof}

\subsection{Simulation}

To validate our asymptotic intervals, we perform Monte Carlo simulation with the same design as in Section~\ref{sec:simdesign} for three sample sizes: $n=200,700$ and $2000$. For each $x\in [-2.5,2.5]$, the estimated coverage probability \emph{at point} $x$ is the proportion of 1000 runs in which our confidence interval at $x$ contains $T_0(x)$. Table~\ref{table:pointwise} displays the summary statistics for the pointwise coverage probabilities and the medians of average widths of the intervals across $x\in[-2.5,2.5]$.

\begin{table}[]
\begin{tabular}{crrcccc}
\hline
\multirow{2}{*}{$1-\alpha$} & \multirow{2}{*}{$n$} & \multirow{2}{*}{$m$} & \multirow{2}{*}{Average width} & \multicolumn{3}{c}{Pointwise coverage probabilities} \\ \cline{5-7} 
                       &                    &                    &                                & Minimum    & Maximum   & Average         \\ \hline
\multirow{3}{*}{0.90}  & 200                & 50                 & 3.35                           & 0.008      & 0.865     & 0.73            \\
                       & 700                & 175                & 2.44                           & 0.362      & 0.881     & 0.83            \\
                       & 2000               & 500                & 1.62                           & 0.774      & 0.889     & \textbf{0.86}   \\ \hline
\multirow{3}{*}{0.95}  & 200                & 50                 & 3.93                           & 0.01       & 0.924     & 0.79            \\
                       & 700                & 175                & 2.84                           & 0.341      & 0.942     & 0.88            \\
                       & 2000               & 500                & 1.94                           & 0.81       & 0.945     & \textbf{0.92}   \\ \hline
\multirow{3}{*}{0.99}  & 200                & 50                 & 4.93                           & 0.003      & 0.979     & 0.84            \\
                       & 700                & 175                & 3.63                           & 0.365      & 0.986     & 0.93            \\
                       & 2000               & 500                & 2.55                           & 0.847      & 0.99      & \textbf{0.97}   \\ \hline
\end{tabular}
\caption{Evaluations of our $(1-\alpha)$-level pointwise confidence intervals of the optimal transport map from $N(0,1)$ to $\operatorname{Gamma}(5,0.5)$ based on 1,000 pairs of samples from each distribution. The table displays the median of average widths and summary of the coverage probabilities of the pointwise confidence intervals over $[-2.5,2.5]$.}
\label{table:pointwise}
\end{table}

 For small sample sizes ($n=200$ and $n=700$), the minimum coverage probabilities are much smaller than the nominal probabilities ($1-\alpha$), whereas the maximum coverage probabilities are close to the nominal probabilities. And as the sample size increases, both the minimum and maximum coverage probabilities converge to the nominal probabilities. Of course, the widths of the confidence intervals decreases as $n$ increases and $1-\alpha$ decreases.

\begin{figure}
    \centering
    \includegraphics[width=0.95\hsize]{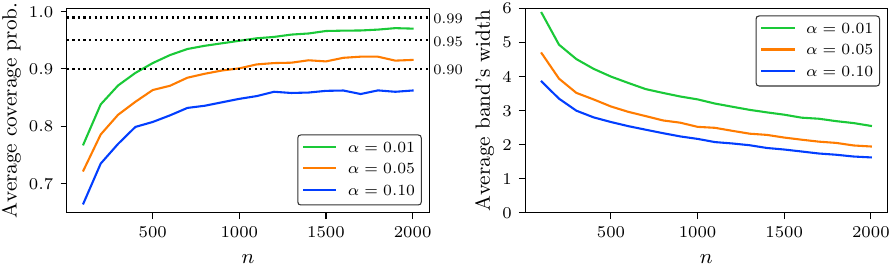}
    \caption{Plots of the average of the coverage probabilities (left) and the median of average widths (right) of the simulated pointwise confidence intervals over $[-2.5,2.5]$ as functions of sample size $n$.}
    \label{fig:pointwise1}
\end{figure}

In Figure~\ref{fig:pointwise1}, we plot the average coverage probability and the median of average width as a function of $n$. From the plots, we see that the average coverage probability converges to the nominal probability as $n$ increases. However, the convergent is very slow for $1-\alpha=0.90$ and $0.95$. Ss discussed in Section~\ref{sec:result}, this is attributed to the insufficient number of sample points from the right tail of $N(0,1)$, which hinders the estimation of the transport map between the two distributions.

\begin{figure}
    \centering
    \includegraphics[width=0.95\hsize]{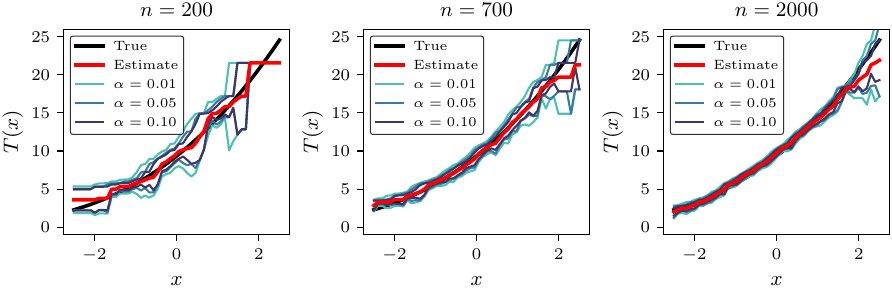}
    \caption{Examples of $(1-\alpha)$-level pointwise confidence intervals over $[-2.5,2.5]$, for three different values of $\alpha$, based on specific samples of size $200, 700$ and $2000$.}
    \label{fig:pointwise2}
\end{figure}

The cause of the sub-optimal coverage probability becomes more apparent when examining the individual confidence intervals (Figure~\ref{fig:pointwise2}). We can see that for larger values of $x$, our plug-in estimator (the red curves) remains significantly distant from the true transport map (the black curves). Consequently, this discrepancy causes the confidence intervals to fail in capturing the transport map accurately.

\section{Mathematical Tools}\label{sec:tools}
\subsection{Functional delta method}
Consider stochastic processes $\mathbb{X}_n$ and $\mathbb{T}$ with values in a normed linear space $\mathcal{D}$ and a function $\phi:\mathcal{D}\to\mathcal{E}$ where $\mathcal{E}$ is another normed linear space. The functional delta method provides a way to turn the weak convergence of a sequence of stochastic processes $a_n(\mathbb{X}_n - \mathbb{T})$ into that of $a_n\left(\phi(\mathbb{X}_n)-\phi(\mathbb{T})\right)$. The idea is to realize that the latter can be written as $a_n\left(\phi(\mathbb{T} + a_n^{-1}h_n)- \phi(\mathbb{T})\right)$ where $h_n=a_n(\mathbb{X}_n - \mathbb{T})$. With sufficient differentiability condition on $\phi$ we expect this sequence to converge to $\phi'(\mathbb{T})$ as $a_n\to \infty$. To rigorously obtaining the convergence, we need the notion of \emph{Hadamard differentiable} functions.

\begin{definition}[Hadamard Derivative]
    A function $\phi:\mathcal{D}\to\mathcal{E}$ is Hadamard differentiable at $\mathbb{T}\in\mathcal{D}$ if there exists a continuous linear map $\phi'_{\mathbb{T}}:\mathcal{D}\to\mathcal{E}$ such that for any $h_t,h\in\mathcal{D}$ satisfying $h_t\to h$ in $\mathcal{D}$ as $t\to 0$, we have
    \begin{align}
        \left\lVert \frac{\phi(\mathbb{T} +th_t) - \phi(\mathbb{T})}{t} - \phi'_{\mathbb{T}}(h) \right\rVert_{\mathcal{E}} \to 0, \quad \text{as }\ \ t\to 0.
    \end{align}
    In this case, we say that $\phi'_{\mathbb{T}}$ is the Hadamard derivative of $\phi$ at $\mathbb{T}$.
\end{definition}
If $\phi$ is Hadamard differentiable then one can obtain the convergence of stochastic processes under $\phi$---this is essentially the statement of the functional delta method.
\begin{lemma}[{\cite[Theorem 20.8]{Vaart1998}}] \label{lemma:functional_delta}
    Let $\mathcal{D}$ and $\mathcal{E}$ be normed linear spaces. Let $\mathbb{X}_n,\mathbb{X}$ and $\mathbb{T}$ be stochastic processes with values in $\mathcal{D}$ such that $a_n(\mathbb{X}_n - \mathbb{T}) \xrightarrow{d} \mathbb{X}$. Let $\phi:\mathcal{D}\to\mathcal{E}$ be Hadamard differentiable at $\mathbb{T}$. Then,
    \begin{align}
        a_n\left(\phi(\mathbb{X}_n)-\phi(\mathbb{T})\right) &\xrightarrow{d} \phi'_{\mathbb{T}}(\mathbb{X}). \label{eq:func_convd} 
        \intertext{If, in addition, $\phi'_{\mathbb{T}}$ is continuous on $\mathcal{D}$. Then,} a_n\left(\phi(\mathbb{X}_n)-\phi(\mathbb{T})\right)  &= \phi'_{\mathbb{T}}\left(a_n(\mathbb{X}_n-\mathbb{T})\right) + o_{\mathbb{P}}(1). \label{eq:func_convP}
    \end{align}
\end{lemma}

\subsection{Functional delta method for bootstrap}\label{sec:funcboot}
In this work, we use the bootstrap procedure to estimate the distribution of stochastic processes of type $a_n\left(\phi(\mathbb{X}_n) - \phi(\mathbb{T})\right)$. Let $\tilde{P}_n$ be the empirical distribution based on the sample $X_1,\ldots,X_n$ and $\mathbb{X}^*_n$ be the bootstrap estimator based on a bootstrap sample $X^*_1,\ldots,X^*_n\sim \tilde{P}_n$. In view of~\eqref{eq:func_convd}, the asymptotic consistency of the bootstrap distribution requires that:
\begin{align}
    a_n\left(\phi(\mathbb{X}_n) - \phi(\mathbb{T})\right) \xrightarrow{d} \phi'_{\mathbb{T}}(\mathbb{X}) \quad \text{and} \quad a_n\left(\phi(\mathbb{X}^*_n) - \phi(\mathbb{X}_n)\right) \xrightarrow{\mathbb{P}\vert \tilde{P}_n} \phi'_{\mathbb{T}}(\mathbb{X}), \label{eq:boot_req}
\end{align}
where $\xrightarrow{\mathbb{P}\vert \tilde{P}_n}$ denotes the convergence in probability, conditionally given $X_1,\ldots,X_n$ in distribution, which can be formally written in terms of the bounded Lipschitz metric:
\begin{equation*}
    \sup_{h\in BL_1(\mathcal{E})}\left\lvert \mathbb{E}_n\left[ h\left(a_n(\phi(\mathbb{X}^*_n) - \phi(\mathbb{X}_n)\right) \right] - \mathbb{E}\left[h\left(\phi'(\mathbb{T})\right)\right] \right\rvert \xrightarrow{\mathbb{P}}  0,
\end{equation*}
where $\mathbb{E}_n$ is the expectation with respect to $\tilde{P}_n$ and $BL_1(\mathcal{E})$ is the space of Lipschitz functions $h:\mathcal{E} \to [-1,1]$ with Lipschitz constants bounded by one.

The first convergence in~\eqref{eq:boot_req} can be obtained via the functional delta method, and the second convergence can be obtained using the following lemma:

\begin{lemma}[{\cite[Theorem 23.9]{Vaart1998}}] \label{lemma:deltaboot}
    Let $\mathcal{D}$ and $\mathcal{E}$ be normed linear spaces. Let $\mathbb{X}_n,\mathbb{X}$ and $\mathbb{T}$ be stochastic processes with values in $\mathcal{D}$ such that $a_n(\mathbb{X}_n - \mathbb{T}) \xrightarrow{d} \mathbb{X}$ and $a_n(\mathbb{X}^*_n - \mathbb{X}_n) \xrightarrow{\mathbb{P}\vert\tilde{P}_n} \mathbb{X}$. Let $\phi:\mathcal{D}\to\mathcal{E}$ be Hadamard differentiable at $\mathbb{T}$. If $\mathbb{X}$ is tight, then  \begin{equation*}
        a_n\left(\phi(\mathbb{X}^*_n) - \phi(\mathbb{X}_n)\right) \xrightarrow{\mathbb{P}\vert \tilde{P}_n} \phi'_{\mathbb{T}}(\mathbb{X}).
    \end{equation*}
\end{lemma}

\section{Full Proof of Theorem \ref{thm:consistency}} \label{sec:proofconsistency}

\begin{proof}[Proof of Theorem \ref{thm:consistency}]
    We divide the proof into seven steps.

    \textbf{Step 1. Find the stochastic limits of $\sqrt{n+m}(\hat{F}_P^*(x) - \hat{F}_P(x))$ and $\sqrt{n+m}(\hat{F}_Q^*(x) - \hat{F}_Q(x))$ in the uniform norm.}
    
    We use the result from~\cite[Theorem 3]{Komls1975} which states that, for any iid sequence $U_1,\ldots \sim P_U$ where $P_U$ is the uniform random distribution on $[0,1]$ and $A_n(t)=n^{-1}\sum_{i=1}^n\mone\{U_i \leq t\}$, there is a Brownian bridge $\mathbb{G}_1$ such that
    \begin{equation}\label{eq:uniconv}
        P_U\left(\sup_{0\leq t \leq 1}\left\lvert \sqrt{n}(A_n(t) - t) - \mathbb{G}_1(t)\right\rvert >\delta\right) \xrightarrow{} 0,
    \end{equation}
    as $n \to \infty$, for any $\delta>0$.
    
    Conditional on $X_1,\ldots,X_n$, $\hat{F}_P$ is a non-random continuous function. So we can apply the above result with $U_i=\hat{F}_P(X^*_i)$ and $t=\hat{F}_P(x)$, yielding $A_n(t)= n^{-1}\sum_{i=1}^n\mone\{U_i \leq \hat{F}_P(x)\} = n^{-1}\sum_{i=1}^n\mone\{\hat{F}_P(X^*_i) \leq \hat{F}_P(x)\} = n^{-1}\sum_{i=1}^n\mone\{X^*_i \leq x\}=\hat{F}_P^*(x)$. The fact that~\eqref{eq:uniconv} holds for all sequences $X_1,\ldots,X_n$ allows us to obtain the $P$-almost sure convergence:
    \begin{equation*}
        \hat{P}_n\left(\sup_{x}\left\lvert \sqrt{n}(\hat{F}_P^*(x) - \hat{F}_P(x)) - \mathbb{G}_1(\hat{F}_P(x))\right\rvert >\delta\right) \xrightarrow{P\text{-a.s.}} 0.
    \end{equation*}
    From~\cite[Lemma A.1]{Horvath2008}, we have $\sup_x \lvert \hat{F}_P(x)-F_P(x)\rvert \xrightarrow{P\text{-a.s.}} 0$. By the uniform continuity of the Brownian bridge, we have $\mathbb{P}\big(\sup_{x}\lvert \mathbb{G}_1(\hat{F}_P(x))
        -\mathbb{G}_1(F_P(x))\rvert >\delta\big) \xrightarrow{P\text{-a.s.}} 0$. Consequently, we have
    \begin{equation}
        \hat{P}_n\left(\sup_{x}\left\lvert \sqrt{n}(\hat{F}_P^*(x) - \hat{F}_P(x)) - \mathbb{G}_1(F_P(x))\right\rvert >\delta\right) \xrightarrow{P\text{-a.s.}} 0.
    \end{equation}
    From the condition $n/(n+m)\to \kappa$ as $n,m\to\infty$, we apply Slutsky's lemma and obtain $(\sqrt{1/\kappa}-\sqrt{(n+m)/n})\sup_x \mathbb{G}_1(F_P(x)) \to 0$ in probability independent of $X_1,\ldots,X_n$. Therefore, we have
    \begin{equation}\label{eq:convFstar}
        \hat{P}_n\left(\sup_{x}\left\lvert \sqrt{n+m}(\hat{F}_P^*(x) - \hat{F}_P(x)) - \sqrt{1/\kappa}\mathbb{G}_1(F_P(x))\right\rvert >\delta\right) \xrightarrow{P\text{-a.s.}} 0.
    \end{equation}
    Using a similar argument, we obtain
    \begin{equation*}
        \hat{Q}_m\left(\sup_{x}\left\lvert \sqrt{n+m}(\hat{F}_Q^*(x) - \hat{F}_Q(x)) - \sqrt{1/(1-\kappa)}\mathbb{G}_2(F_Q(x))\right\rvert >\delta\right) \xrightarrow{Q\text{-a.s.}} 0,
    \end{equation*}
    where $\mathbb{G}_2$ is another independent Browninan bridge.
    
    \textbf{Step 2. Convert the convergences in the uniform norm to those in the bounded Lipschitz metric.}
    
    Recall that $BL_1(\ell^{\infty}(\mathcal{F}))$ denotes the space of Lipschitz functions $h:\ell^{\infty}(\mathcal{F}) \to [-1,1]$ with Lipschitz constants bounded by one. 
    For any $h\in BL_1(\ell^{\infty}(\mathcal{F}))$, we have $\lvert h(\sqrt{n+m}(\hat{F}_P^* - \hat{F}_P)) - h(\sqrt{1/\kappa}\mathbb{G}_1\circ F_P) \rvert \leq \sup_{x}\lvert \sqrt{n+m}(\hat{F}_P^*(x) - \hat{F}_P(x)) - \sqrt{1/\kappa} \mathbb{G}_1(F_P(x))\rvert$. Consequently, it  holds that
    \begin{align}
        &\hat{P}_n\left(\sup_{x}\left\lvert \sqrt{n+m}(\hat{F}_P^*(x) - \hat{F}_P(x)) - \sqrt{1/\kappa}\mathbb{G}_1(F_P(x))\right\rvert >\delta\right) \\
        &\quad  \geq \hat{P}_n\left(\lvert h(\sqrt{n+m}(\hat{F}_P^* - \hat{F}_P)) - h(\sqrt{1/\kappa}\mathbb{G}_1\circ F_P)\rvert >\delta\right).
    \end{align}
    Define an event $S_{\delta,n,m} = \{ h(\sqrt{n+m}(\hat{F}_P^* - \hat{F}_P)) - h(\sqrt{1/\kappa}\mathbb{G}_1\circ F_P) > \delta\}$. Denoting by $\mathbb{E}_n$ the expectation with respect to $\hat{P}_n$, Jensen's inequality yields that
    \begin{align}
        &\left\lvert \mathbb{E}_n h(\sqrt{n+m}(\hat{F}_P^* - \hat{F}_P)) - \mathbb{E} h(\sqrt{1/\kappa}\mathbb{G}_1\circ F_P) \right\rvert \\
        &\leq \mathbb{E}_n \mathbb{E} \left\lvert h(\sqrt{n+m}(\hat{F}_P^* - \hat{F}_P)) - h(\sqrt{1/\kappa}\mathbb{G}_1\circ F_P) \right\rvert \\
        &\leq 2\hat{P}_n(S_{\delta,n,m})+\delta \hat{P}_n(S^c_{\delta,n,m}) \\
        &\leq 2\hat{P}_n(S_{\delta,n,m}) + \delta,
    \end{align} 
    where we used the fact that $h$ is bounded by one in the second to last inequality. From~\eqref{eq:convFstar}, we have $\hat{P}_n(S_{\delta,n,m}) \to 0$ for any arbitrary $\delta$; therefore, by taking supremum over all $h\in BL_1(\ell^{\infty}(\mathcal{F}))$, there is a Brownian bridge $\mathbb{G}_1$ such that
    \begin{equation}\label{eq:EhF}
        \sup_{h\in BL_1(\ell^{\infty}(\mathcal{F}))}\left\lvert \mathbb{E}_n h(\sqrt{n+m}(\hat{F}_P^* - \hat{F}_P)) - \mathbb{E} h(\sqrt{1/\kappa}\mathbb{G}_1\circ F_P) \right\rvert \xrightarrow{P\text{-a.s.}} 0.
    \end{equation}
    Similarly, letting $\mathbb{E}_m$ be the expectation with respect to $\hat{Q}_m$, there is a Brownian bridge $\mathbb{G}_2$ such that
    \begin{equation}\label{eq:EhG}
        \sup_{h\in BL_1(\ell^{\infty}(\mathcal{F}))}\left\lvert \mathbb{E}_m h(\sqrt{n+m}(\hat{F}_Q^* - \hat{F}_Q)) - \mathbb{E} h(\sqrt{1/(1-\kappa)}\mathbb{G}_2\circ F_Q) \right\rvert \xrightarrow{Q\text{-a.s.}} 0.
    \end{equation}

    \textbf{Step 3. Prove the convergence of the sequence $\sqrt{n+m}(\hat{F}_P^* - \hat{F}_P, \hat{F}_Q^*-\hat{F}_Q)$ in the bounded Lipschitz metric.}
    
    Now we will show that 
    \begin{equation}\label{eq:bootconv}
        \sup_{h\in BL_1(\ell^{\infty}(\mathcal{F})^2)}\left\lvert \mathbb{E}_{n,m} h\big(\sqrt{n+m}(\hat{F}_P^* - \hat{F}_P, \hat{F}_Q^*-\hat{F}_Q)\big) - \mathbb{E} h\big(\sqrt{1/\kappa}\mathbb{G}_1\circ F_P, \sqrt{1/(1-\kappa)}\mathbb{G}_2\circ F_Q\big) \right\rvert \xrightarrow{P\otimes Q\text{-a.s.}} 0.
    \end{equation} 
    Consider any arbitrary $h\in BL_1(\ell^{\infty}(\mathcal{F})^2)$.
    For any $w\in \ell^{\infty}$, define $h_1^w,h_2^w\in BL_1(\ell^{\infty})$ by $h_1^w(\cdot) = h(\cdot, w)$ and $h_2^w(\cdot) = h(w,\cdot)$. Denoting $E_{n,m}$ the expectation with respect to the empirical distribution $\hat{\Pr}_{n,m}$, we have
    \begin{align}
        &\left\lvert \mathbb{E}_{n,m} h\big(\sqrt{n+m}(\hat{F}_P^* - \hat{F}_P, \hat{F}_Q^*-\hat{F}_Q)\big) - \mathbb{E}h\big(\sqrt{1/\kappa}\mathbb{G}_1\circ F_P, \sqrt{1/(1-\kappa)}\mathbb{G}_2\circ F_Q\big) \right\rvert  \\
        &\leq  \left\lvert \mathbb{E}_{n,m} h\big(\sqrt{n+m}(\hat{F}_P^* - \hat{F}_P, \hat{F}_Q^*-\hat{F}_Q)\big) - \mathbb{E}_{n,m}\mathbb{E}h\big(\sqrt{1/\kappa}\mathbb{G}_1\circ F_P, \sqrt{n+m}(\hat{F}_Q^* - \hat{F}_Q)\big) \right\rvert \\
        & \qquad + \left\lvert \mathbb{E}_{n,m}\mathbb{E} h\big(\sqrt{1/\kappa}\mathbb{G}_1\circ F_P,\sqrt{n+m}(\hat{F}_Q^* - \hat{F}_Q)\big) - \mathbb{E}h\big(\sqrt{1/\kappa}\mathbb{G}_1\circ F_P,\sqrt{1/(1-\kappa)}\mathbb{G}_2\circ F_Q\big) \right\rvert \\
        &\leq  \sup_{w\in \ell^{\infty}(\mathcal{F})}\left\lvert \mathbb{E}_n h_1^w\big(\sqrt{n+m}(\hat{F}_P^* - \hat{F}_P)\big) - \mathbb{E}h_1^w(\sqrt{1/\kappa}\mathbb{G}_1\circ F_P) \right\rvert   \\
        & \qquad +  \sup_{w\in \ell^{\infty}(\mathcal{F})}\left\lvert \mathbb{E}_m h_2^w\big(\sqrt{n+m}(\hat{F}_Q^* - \hat{F}_Q)\big) - \mathbb{E}h_2^w(\sqrt{1/(1-\kappa}\mathbb{G}_2\circ F_Q) \right\rvert   \\
        &\leq  \sup_{h_1\in BL_1(\ell^{\infty}(\mathcal{F}))}\left\lvert \mathbb{E}_n h_1\big(\sqrt{n+m}(\hat{F}_P^* - \hat{F}_P)\big) - \mathbb{E}h_1(\sqrt{1/\kappa}\mathbb{G}_1\circ F_P) \right\rvert  \\
        & \qquad + \sup_{h_2\in BL_1(\ell^{\infty}(\mathcal{F}))}\left\lvert \mathbb{E}_m h_2\big(\sqrt{n+m}(\hat{F}_Q^* - \hat{F}_Q)\big) - \mathbb{E}h_2(\sqrt{1/(1-\kappa)}\mathbb{G}_2\circ F_Q) \right\rvert.
    \end{align}
    Using~\eqref{eq:EhF} and~\eqref{eq:EhG}, the last expression converges $P\otimes Q$-a.s. to zero uniformly over $h\in BL_1(\ell^{\infty}(\mathcal{F})^2)$ as $n,m\to \infty$. Thus taking the supremum over such $h$ yields~\eqref{eq:bootconv} as desired.

    In addition, under the required conditions on $f_P,f_Q$ and $K$, we can use Corollary 2 of~\cite{Gin2007} to obtain $\sqrt{n}(\hat{F}_P - F_P) \xrightarrow{d} \mathbb{G}_1 \circ F_P$ and $\sqrt{m}(\hat{F}_Q - F_Q) \xrightarrow{d} \mathbb{G}_2 \circ F_Q$. Consequently, we obtain
    \begin{equation}\label{eq:smoothconv}
    \sqrt{n+m}(\hat{F}_P - F_P, \hat{F}_Q - F_Q) \xrightarrow{d} (\sqrt{1/\kappa}\mathbb{G}_1\circ F_P,\sqrt{1/(1-\kappa)}\mathbb{G}_2\circ F_Q).
    \end{equation}

    \textbf{Step 4. Apply the delta method for bootstrap.}

    With~\eqref{eq:bootconv} and~\eqref{eq:smoothconv}, we can now apply the delta method for bootstrap (see Appendix~\ref{sec:funcboot} for a brief exposition). Define a functional $\Psi:D[a,b]\times D[a,b] \to D[a,b]$
    \begin{align}
        \Psi(u,v) = (v^{-1}\circ u(\cdot))/s_\kappa(\cdot).
    \end{align}
     We recall from~\eqref{eq:derivative} that the Hadamard derivative of $\Psi$ at $(F_P,F_Q)$ is 
     \begin{align}
         \Psi'(u,v) = (u(\cdot) - v \circ T_0(\cdot))/((f_Q\circ T)(\cdot) s_\kappa (\cdot)).
     \end{align}
    Using Lemma~\ref{lemma:deltaboot} on $\Psi$, the sequence 
    \begin{align}
        \sqrt{n+m}(\Psi(\hat{F}^*_{P,n},\hat{F}^*_{Q,m}) - \Psi(\hat{F}_P,\hat{F}_Q)) = \sqrt{n+m}(\hat{T}_{n,m}^*-\hat{T}_{n,m})(\cdot)/s_\kappa(\cdot)
    \end{align}
    weakly converges  to 
    \begin{align}
        &\Psi'(\sqrt{1/\kappa}\mathbb{G}_1\circ F_P, \sqrt{1/(1-\kappa)}\mathbb{G}_2\circ F_Q)\\
        &=\left(\sqrt{1/\kappa}\mathbb{G}_1\circ F_P - \sqrt{1/(1-\kappa)}\mathbb{G}_2\circ F_Q\circ T_0\right)/((f_Q\circ T_0)s_\kappa )\\
        &=\left(\sqrt{1/\kappa}\mathbb{G}_1\circ F_P - \sqrt{1/(1-\kappa)}\mathbb{G}_2\circ F_P\right)/((f_Q\circ T_0)s_\kappa)\\
        &=\mathsf{Z}_\kappa,
    \end{align}
    under the bounded Lipschitz metric. More precisely, denoting $\hat{\mathsf{Z}}_{n,m}\coloneqq \sqrt{n+m}(\hat{T}_{n,m}^*-\hat{T}_{n,m})/s_\kappa$, we have 
    \begin{equation}\label{eq:Eh}
        \sup_{h\in BL_1(\ell^{\infty}(\mathcal{F}))}\left\lvert \mathbb{E}_{n,m}h(\hat{\mathsf{Z}}_{n,m}) - \mathbb{E}h(\mathsf{Z}_\kappa) \right\rvert \xrightarrow{P\otimes Q}  0,
    \end{equation}

    \textbf{Step 5. Prove that the sequence $\hat{\Pr}_{n,m}\Big(\sup_x\big\lvert\hat{\mathsf{Z}}_{n,m}(x)\big\rvert\leq x_0\Big)$ converges to $\mathbb{P}\Big(\sup_x\lvert \mathsf{Z}_\kappa(x) \rvert \leq x_0\Big)$ in $P\otimes Q$-probability.}
        
    From~\eqref{eq:Eh}, we consider the supremum over the set of functions $h$ that are of the form $h(f) = \tilde{h}(\sup_x\lvert f(x) \rvert)$ for some $\tilde{h}\in BL_1[0,\infty)$, which results in
    \begin{equation*}
        \sup_{\tilde{h}\in BL_1[0,\infty)}\Big\lvert \mathbb{E}_{n,m}\big[ \tilde{h}\big(\sup_x\big\lvert \hat{\mathsf{Z}}_{n,m}(x)\big\rvert\big) \big] - \mathbb{E}\big[\tilde{h}\big(\sup_x\lvert \mathsf{Z}_\kappa(x) \rvert \big)\big] \Big\rvert \xrightarrow{P\otimes Q}  0,
    \end{equation*}
    By scaling, we can extend the supremum to $BL_M[0,\infty)$, the set of bounded Lipschitz functions with the Lipschitz constants bounded by $M>0$. Given any $\epsilon$ and $\delta>0$, there exists $n,m$ large enough so that
    \begin{equation*}
        P\Big(\sup_{\tilde{h}\in BL_M[0,\infty)}\Big\lvert \mathbb{E}_{n,m}\big[ \tilde{h}\big(\sup_x\big\lvert\hat{\mathsf{Z}}_{n,m}(x)\big\rvert\big) \big] - \mathbb{E}\big[\tilde{h}\big(\sup_x\lvert \mathsf{Z}_\kappa(x) \rvert \big)\big] \Big\rvert > \delta/2 \Big) \leq  \epsilon.
    \end{equation*}
    Fix $x_0>0$. For a sufficiently large $M$, we can find two functions $\tilde{h}_L,\tilde{h}_U \in BL_M[0,\infty)$ such that $\tilde{h}_L(x) \leq \mone\{ 0 \leq x \leq x_0 \} \leq \tilde{h}_U(x)$ and $\mathbb{E}\big[\tilde{h}_U\big(\sup_x\lvert \mathsf{Z}_\kappa(x) \rvert \big) - \tilde{h}_L\big(\sup_x\lvert \mathsf{Z}_\kappa(x) \rvert \big)\big] < \delta/2$. It follows that, with a probability greater than $1-\epsilon$,
    \begin{align}
        &\hat{\Pr}_{n,m}\Big(\sup_x\big\lvert\hat{\mathsf{Z}}_{n,m}(x)\big\rvert\leq x_0\Big) - \mathbb{P}\Big(\sup_x\lvert \mathsf{Z}_\kappa(x) \rvert \leq x_0\Big) \\
        &\leq \mathbb{E}_{n,m}\big[ \tilde{h}_U\big(\sup_x\big\lvert\hat{\mathsf{Z}}_{n,m}(x)\big\rvert\big) \big] - \mathbb{E}\big[\tilde{h}_L\big(\sup_x\lvert \mathsf{Z}_\kappa(x) \rvert \big)\big] \\
        &\leq \mathbb{E}\big[\tilde{h}_U\big(\sup_x\lvert \mathsf{Z}_\kappa(x) \rvert \big)\big] - \mathbb{E}\big[\tilde{h}_L\big(\sup_x\lvert \mathsf{Z}_\kappa(x) \rvert \big)\big] + \delta/2 \\
        &\leq \delta/2 + \delta/2 = \delta.
    \end{align}
    Similarly, we have
    \begin{equation*}
        \hat{\Pr}_{n,m}\Big(\sup_x\big\lvert\hat{\mathsf{Z}}_{n,m}(x)\big\rvert\leq x_0\big) - \mathbb{P}\big(\sup_x\lvert \mathsf{Z}_\kappa(x) \rvert \leq x_0\Big) \geq   -\delta.
    \end{equation*}
    Therefore, we have the convergence in $P\otimes Q$-probability for each $x_0 >0$:
    \begin{equation*}
        \hat{\Pr}_{n,m}\Big(\sup_x\big\lvert\hat{\mathsf{Z}}_{n,m}(x)\big\rvert\leq x_0\Big) \xrightarrow{P\otimes Q} \mathbb{P}\Big(\sup_x\lvert \mathsf{Z}_\kappa(x) \rvert \leq x_0\Big).
    \end{equation*}
    
    \textbf{Step 6. Use Lemma~\ref{lemma:sestimate} to approximate $s_\kappa$ by $\hat{s}_{n,m}$ and finish the proof of~\eqref{eq:empconv}.} 
    
    Let $\hat{\mathsf{Z}}^*_{n,m}\coloneqq \sqrt{n+m}(\hat{T}_{n,m}^*-\hat{T}_{n,m})/\hat{s}_{n,m}$. So we have to show that
    \begin{equation}\label{eq:Zconv}
        \hat{\Pr}_{n,m}\Big(\sup_x\big\lvert\hat{\mathsf{Z}}^*_{n,m}(x)\big\rvert\leq x_0\Big) \xrightarrow{P\otimes Q} \mathbb{P}\Big(\sup_x\lvert \mathsf{Z}_\kappa(x) \rvert \leq x_0\Big).
    \end{equation}
    First, observe that
    \begin{align}
        \hat{\Pr}_{n,m}\Big(\sup_x\big\lvert\hat{\mathsf{Z}}^*_{n,m}(x)\big\rvert\leq x_0\Big) &= \hat{\Pr}_{n,m}\left(\sup_x\big\lvert\hat{\mathsf{Z}}_{n,m}(x)\big\rvert \cdot \left\lvert \frac{s_\kappa(x)}{\hat{s}_{n,m}(x)} \right\rvert \leq x_0\right).
    \end{align}
    This is where we use Lemma~\ref{lemma:sestimate}. For arbitrarily small $\epsilon,\delta,\delta'>0$, there are sufficiently large $n,m$ such that, with $P\otimes Q$-probability greater than $1-\epsilon$, the following inequalities hold simulteneously: first,
    \begin{align}
        \hat{\Pr}_{n,m}\Big(\sup_x\big\lvert\hat{\mathsf{Z}}_{n,m}(x)\big\rvert  \leq x_0(1-\delta)\Big) &\leq\hat{\Pr}_{n,m}\Big(\sup_x\big\lvert\hat{\mathsf{Z}}^*_{n,m}(x)\big\rvert \leq x_0\Big) \\
        &\leq \hat{\Pr}_{n,m}\Big(\sup_x \big\lvert\hat{\mathsf{Z}}_{n,m}(x)\big\rvert \leq x_0(1+\delta)\Big), 
    \end{align}
    which is a result of~\eqref{eq:sconv}.
    Second, we have 
    \begin{align}
        \left\lvert \hat{\Pr}_{n,m}\Big(\sup_x\big\lvert\hat{\mathsf{Z}}_{n,m}(x)\big\rvert  \leq x_0(1-\delta')\Big) - \mathbb{P}\Big(\sup_x\lvert \mathsf{Z}_\kappa(x) \rvert \leq x_0(1-\delta')\Big) \right\rvert &< \delta/2,
        \intertext{and lastly,}
        \left\lvert \hat{\Pr}_{n,m}\Big(\sup_x\big\lvert\hat{\mathsf{Z}}_{n,m}(x)\big\rvert  \leq x_0(1+\delta')\Big) - \mathbb{P}\Big(\sup_x\lvert \mathsf{Z}_\kappa(x) \rvert \leq x_0(1+\delta')\Big) \right\rvert &< \delta/2.
    \end{align}
    By the continuity of the distribution function of $\sup_x\lvert \mathsf{Z}_\kappa(x) \rvert$ (Lemma~\ref{lemma:cont}), we can choose $\delta'$ sufficiently small that $\left\lvert \mathbb{P}\Big(\sup_x\lvert \mathsf{Z}_\kappa(x) \rvert \leq x_0(1-\delta')\Big) - \mathbb{P}\Big(\sup_x\lvert \mathsf{Z}_\kappa(x) \rvert \leq x_0(1-\delta')\Big) \right\rvert < \delta/2$. Combining the above inequalities yields the following bound with $P\otimes Q$-probability greater than $1-\epsilon$, 
    \[\left\lvert \hat{\Pr}_{n,m}\Big(\sup_x\big\lvert\hat{\mathsf{Z}}^*_{n,m}(x)\big\rvert \leq x_0\Big) - \mathbb{P}\Big(\sup_x\lvert \mathsf{Z}_\kappa(x) \rvert \leq x_0\Big)  \right\rvert < \delta, \] 
    which implies~\eqref{eq:Zconv}.

    \textbf{Step 7. Prove~\eqref{eq:estconv} using the functional delta method.}
    
    Starting with~\eqref{eq:smoothconv}, it follows from the functional delta method (Lemma~\ref{lemma:functional_delta}) that
        \begin{align}
        \sqrt{n+m}\left(\frac{\hat{T}_{n,m} - T_0}{s_\kappa}\right) &= \sqrt{n+m}(\psi(\hat{F}_P,\hat{F}_Q)-\psi(F_P,F_Q)) \\  &\xrightarrow{d} \psi'_{F,G}\left(\sqrt{1/\kappa}\mathbb{G}_1\circ F_P,\sqrt{1/(1-\kappa)}\mathbb{G}_2\circ F_Q\right) \\
        &= \mathsf{Z}_\kappa.
    \end{align}
    From~\eqref{eq:sconv}, we have $s_\kappa/\hat{s}_{n,m} \xrightarrow{P\otimes Q} 1$ under the uniform norm, so it follows from Slutsky's theorem that
    \[ \sqrt{n+m}\left(\frac{\hat{T}_{n,m} - T_0}{\hat{s}_{n,m}}\right) \xrightarrow{d} \mathsf{Z}_\kappa.\]
    The continuous mapping theorem yields
        \begin{equation*}
        \sqrt{n+m}\sup_x \left\lvert\frac{\hat{T}_{n,m}(x) - T_0(x)}{\hat{s}_{n,m}}\right\rvert \xrightarrow{d} \sup_x \lvert \mathsf{Z}_\kappa(x) \rvert,
    \end{equation*}
    as desired.
\end{proof}

\section{Additional Results and Proofs} \label{sec:additionalproofs}

\begin{proof}[Proof of Proposition \ref{eq:smoothbahadur}]
    The proof is analogous to that of Proposition~\ref{prop:error}. Assumption~\ref{assumption:FGK} and~\ref{assumption:gK} guarantee that $F_Q$ satisfies the conditions in Proposition~\ref{prop:error}. In addition, under Assumption~\ref{assumption:FGK}, we have the following weak convergences under the uniform norm which is due to~\cite[Corollary 2 ]{Gin2007}: 
    \[\sqrt{n}(\hat{F}_P - F_P) \xrightarrow{d} \mathbb{G}_1 \circ F_P \quad \text{and} \quad  \sqrt{m}(\hat{F}_Q - F_Q) \xrightarrow{d} \mathbb{G}_2 \circ F_Q,\] 
    for some Brownian bridges $\mathbb{G}_1$ and $\mathbb{G}_2$. This allows us to apply the functional delta method. The rest of the proof follows as in the proof of Proposition~\ref{prop:error}.
\end{proof}

\begin{proof}[Proof of Lemma \ref{lemma:sestimate}]
     Denote 
     \begin{align}
         \hat{A}_{n,m} = \sqrt{(n+m)(n^{-1}+m^{-1})\hat{F}_P(1-\hat{F}_P)}
     \end{align}
     and
     \begin{align}
         A_\kappa = \sqrt{(\kappa^{-1}+ (1-\kappa)^{-1})F_P(1-F_P)}
     \end{align}
     for notational convenience. Since $f_Q$ is uniformly bounded above and $A_\kappa$ is uniformly bounded away from zero on $[a,b]$, the function $s$ is uniformly bounded away from zero on $[a,b]$, so it suffices to show that $\sup_{x\in [a,b]} \lvert \hat{s}_{n,m}(x) - s_\kappa(x)\rvert \xrightarrow{P\otimes Q} 0$.  Now we split the difference into two terms:
     \begin{equation}\label{eq:ssplit}
         \hat{s}_{n,m}(x) - s_\kappa(x) = \left[\frac{\hat{A}_{n,m}(x)}{\hat{f}_Q(\hat{T}_{n,m}(x))} -\frac{\hat{A}_{n,m}(x)}{f_Q(T_0(x))}  \right] + \left[ \frac{\hat{A}_{n,m}(x) - A_\kappa(x)}{f_Q(T_0(x))}  \right].
     \end{equation}
     For the first term in~\eqref{eq:ssplit}, we write \[\frac{\hat{A}_{n,m}(x)}{\hat{f}_Q(\hat{T}_{n,m}(x))} -\frac{\hat{A}_{n,m}(x)}{f_Q(T_0(x))} = \frac{\hat{A}_{n,m}(x)\left[f_Q(T_0(x)) - \hat{f}_Q(\hat{T}_{n,m}(x))\right]}{f_Q(T_0(x))\hat{f}_Q(\hat{T}_{n,m}(x))}.\] 
     The conditions on $K$, $r_n$ and the uniform continuity of $f_Q$ on $[a,b]$ imply $\sup_{x\in\R} \lvert \hat{f}_Q (x) - f_Q(x) \rvert  \xrightarrow{Q\text{-a.s.}} 0$ (see \cite{Bertrand1978} and~\cite[p. 72]{Silverman2018}). Consequently, $\hat{f}_Q\circ\hat{T}_{n,m}$ is uniformly bounded away from zero $P\otimes Q$-a.s. Combining this with a simple observation that $\hat{A}_{n,m}$ is uniformly bounded above, it suffices to show that
     \begin{equation}\label{eq:grecip}
         \sup_{x\in [a,b]} \left \lvert \hat{f}_Q(\hat{T}_{n,m}(x)) -f_Q(T_0(x)) \right \rvert \xrightarrow{P\otimes Q\text{-a.s.}} 0.
     \end{equation}
         We split $\hat{f}_Q\circ \hat{T}_{n,m} - f_Q\circ T_0$ as follows:
    \begin{equation}\label{eq:gTdecom}
        \hat{f}_Q\circ\hat{T}_{n,m} -   f_Q\circ T_0= \{ \hat{f}_Q\circ\hat{T}_{n,m} - f_Q\circ\hat{T}_{n,m} \} + \{ f_Q\circ\hat{T}_{n,m} - f_Q\circ T_0 \} 
    \end{equation}
    For the second term on the right, we use the Bahadur representation (Proposition~\ref{prop:error}) on $[a,b]$, which allows us to obtain the convergence in probability $\hat{T}_{n,m} \xrightarrow{P\otimes Q} T_0$ of functions in $C[a,b]$ under the uniform norm. Since the mapping $T_0 \mapsto g\circ T_0$ is continuous in $C[a,b]$, it follows from the continuous mapping theorem that $f_Q\circ \hat{T}_{n,m} \xrightarrow{P\otimes Q} f_Q\circ T_0$. In other words,
    \begin{equation*}
        \sup_{x\in [a,b]} \lvert f_Q(\hat{T}_{n,m}(x)) - f_Q(T_0(x)) \rvert \xrightarrow{P\otimes Q} 0.
    \end{equation*}
    For the first term , we use $\sup_{x\in\R} \lvert \hat{f}_Q (x) - f_Q(x) \rvert  \xrightarrow{Q\text{-a.s.}} 0$ which implies the convergence $\sup_{x\in[a,b]} \lvert \hat{f}_Q (\hat{T}_{n,m}(x)) - f_Q(\hat{T}_{n,m}(x)) \rvert  \xrightarrow{P\otimes Q\text{-a.s.}} 0$.
    
     Now we consider the second term in~\eqref{eq:ssplit}. From~\cite[Lemma A.1]{Horvath2008}, we have $\sup_x \lvert \hat{F}_P(x)-F_P(x)\rvert \xrightarrow{P\text{-a.s.}} 0$. From this, we split the numerator as follows:
    \begin{alignat*}{2}
        \hat{A}_{n,m} -  A_\kappa   &= \frac1{\hat{A}_{n,m}+A_\kappa}\Big[ (&&\kappa^{-1} + (1-\kappa)^{-1})(\hat{F}_P - F_P) \\
        &  &&+ (\kappa^{-1} + (1-\kappa)^{-1})(F_P - \hat{F}_P)(F_P + \hat{F}_P) \\
        & &&+\left\{(n+m)(n^{-1}+m^{-1}) - (\kappa^{-1} + (1-\kappa)^{-1}) \right\}\hat{F}_P(1-\hat{F}_P) \Big].
    \end{alignat*}
    On the closed interval $[a,b]$, the sequence $\hat{A}_{n,m}+A_\kappa$ is uniformly bounded away from zero, and $F_P + \hat{F}_P$ and $\hat{F}_P(1-\hat{F}_P)$ are uniformly bounded above. Combining these with the fact that $f_Q\circ T$ is also uniformly bounded away from zero on $[a,b]$, we obtain
    \begin{equation}\label{eq:numer}
        \sup_{x\in [a,b]} \left\lvert \frac{\hat{A}_{n,m}(x) - A_\kappa(x)}{f_Q(T_0(x))} \right\rvert \xrightarrow{P\text{-a.s.}} 0.
    \end{equation}
\end{proof}

We also need a result on the continuity of the distribution function of the suprema of Gaussian processes.
\begin{lemma}\label{lemma:cont}
    Let $\mathsf{Z} = \{\mathsf{Z}(t) : t\in I\}$ be a tight Gaussian process with $\mathbb{E}\mathsf{Z}(t) = 0$ and $\mathbb{E}\mathsf{Z}^2(t) = 1$ for all $t\in I$. Then the distribution function $x_0 \mapsto \mathbb{P}\left( \sup_t \lvert \mathsf{Z}(t) \rvert \leq x_0 \right)$ is continuous.
\end{lemma}
\begin{proof}[Proof of Lemma \ref{lemma:cont}]
    We use an anti-concentration inequality for the suprema of Gaussian processes~\cite[Lemma A.1]{Chernozhukov2014}:
    \begin{equation*}
        \sup_{x_0\in\R} \mathbb{P}\left(\left\lvert \sup_{t\in I} \lvert \mathsf{Z}(t) \rvert -x_0 \right\rvert \leq \epsilon \right) \leq 4\epsilon\left( 1 + \mathbb{E}\left[\sup_{t\in I} \lvert \mathsf{Z}(t) \rvert\right]\right).
    \end{equation*}
    which implies the continuity of the distribution function at any $x_0\in\R$.
\end{proof}

\begin{proof}[Proof of Corollary \ref{cor:validity_band}]
    For notational convenience, denote $\lVert \mathsf{Z}_\kappa \rVert_{[a,b]} = \sup_{x\in [a,b]}\lvert \mathsf{Z}_\kappa(x)\rvert$. Recall from~\eqref{eq:empconv} that the sequence $\hat{\Pr}_{n,m}\Big(\sqrt{n + m}\sup_x\big\lvert\hat{T}_{n,m}^*(x)-\hat{T}_{n,m}(x)\big\rvert/\hat{s}_{n,m}(x)\leq x_0\Big)$ converges in $P\otimes Q$-probability to $\mathbb{P}\left(\lVert \mathsf{Z}_\kappa \rVert_{[a,b]} \leq x_0\right)$. By passing along any subsequence and a further subsequence, we can assume that the convergence is $P\otimes Q$-almost surely. 
    
    The almost-sure convergence of the distribution functions implies the almost-sure convergence of the corresponding quantile functions~\cite[Lemma 21.2]{Vaart1998}. Thus, the sequence of quantile functions $\hat{q}_{n,m}$ of $\sqrt{n + m}\sup_x\big\lvert\hat{T}_{n,m}^*(x)-\hat{T}_{n,m}(x)\big\rvert/\hat{s}_{n,m}(x)$ converges $P\otimes Q$-a.s. to the quantile function $q_{\mathsf{Z}}$ of $\lVert \mathsf{Z}_\kappa \rVert_{[a,b]}$ at the continuity points of $q_{\mathsf{Z}}$. Let $1-\alpha$ be one of those continuity points. Then, $\hat{q}_{n,m}(1-\alpha)$ converges $P\otimes Q$-a.s. to $q_{\mathsf{Z}}(1-\alpha)$, and it follows from Slutsky's theorem that $\sqrt{n + m}\sup_x\big\lvert\hat{T}_{n,m}(x)-T_{n,m}(x)\big\rvert/\hat{s}_{n,m}(x) - \hat{q}_{n,m}(1-\alpha)$ converges in distribution to $\lVert \mathsf{Z}_\kappa \rVert_{[a,b]} - q_{\mathsf{Z}}(1-\alpha)$. Therefore,
    \begin{equation}\label{eq:asympconv}
    \begin{split}
     &\mathbb{P}\left(T_0(x) \in \left[\hat{T}_{n,m}(x) \pm \hat{s}_{n,m}(x)\hat{q}_{n,m}(1-\alpha)/\sqrt{n+m}\right], \forall x\in [a,b]\right) \\
    &= \underbrace{\mathbb{P}\left(\sqrt{n + m}\sup_x\frac{\big\lvert\hat{T}_{n,m}(x)-T_{n,m}(x)\big\rvert}{\hat{s}_{n,m}(x)} \leq \hat{q}_{n,m}(1-\alpha)\right)}_{\hat{\mathsf{R}}_{n,m}(1-\alpha)} 
    \\
    &\to \mathbb{P}\left(\lVert \mathsf{Z}_\kappa \rVert_{[a,b]} \leq q_{\mathsf{Z}}(1-\alpha)\right)  = 1-\alpha,
    \end{split}
    \end{equation}
    where the last equality holds because of the continuity of the distribution function of $\lVert \mathsf{Z}_\kappa \rVert_{[a,b]}$ (Lemma~\ref{lemma:cont}). With~\eqref{eq:asympconv}, we will show that the convergence~\eqref{eq:asympconv} holds for all $\alpha \in (0,1)$. As $q_{\mathsf{Z}}(1-\alpha)$ is a monotone function of $\alpha$, it has only countably many discontinuities. Consequently, given any $\alpha \in (0,1)$ and any $\epsilon>0$, there exist continuity points $\alpha_L,\alpha_U \in (0,1)$ such that $\alpha_L < \alpha < \alpha_U$ and $\max\{\alpha_U-\alpha, \alpha-\alpha_L\} < \epsilon/2$. Since the convergence~\eqref{eq:asympconv} holds for $\alpha_L$ and $\alpha_U$, we can find sufficiently large $n$ and $m$ such that $\left\lvert\hat{\mathsf{R}}_{n,m}(1-\alpha_L) - (1-\alpha_L)\right\rvert < \epsilon/2$ and $\left\lvert\hat{\mathsf{R}}_{n,m}(1-\alpha_U) - (1-\alpha_U)\right\rvert < \epsilon/2$. So we must have
    \begin{equation*}
        \hat{\mathsf{R}}_{n,m}(1-\alpha) \leq \hat{\mathsf{R}}_{n,m}(1-\alpha_L) \leq 1-\alpha_L +\epsilon/2 \leq 1-\alpha + \epsilon,
    \end{equation*}
    and similarly,
    \begin{equation*}
        \hat{\mathsf{R}}_{n,m}(1-\alpha) \geq \hat{\mathsf{R}}_{n,m}(1-\alpha_U) \geq 1-\alpha_U -\epsilon/2 \geq 1-\alpha - \epsilon,
    \end{equation*}
    which shows that $\hat{\mathsf{R}}_{n,m}(1-\alpha) \to 1-\alpha$ for all $\alpha\in (0,1)$ as desired.
\end{proof}

\paragraph{Acknowledgements}

We would like to show our gratitude to the associate editor and reviewers for their fruitful comments and suggestions.

\bibliographystyle{alpha}
\bibliography{ConfBand_1D_OT}

\end{document}